\documentclass{amsart}

\usepackage[mathscr]{euscript}
\usepackage{amssymb}
\usepackage{amsmath}
\usepackage{latexsym}
\usepackage{graphicx}
\usepackage{float}

\makeatletter
\renewcommand\@biblabel[1]{#1.}
\makeatother

\newtheorem{thm}{Theorem}[section]
\newtheorem{lem}[thm]{Lemma}
\newtheorem{cor}[thm]{Corollary}

\newtheorem{prop}[thm]{Proposition}
\newtheorem{df}[thm]{Definition}

\newtheorem{remark}[thm]{Remark}
\newtheorem{question}[thm]{Question}

\newcommand\mf{\mathfrak }

\newcommand{\R}{{\mathbb R}}

\newcommand{\Fin}{\textrm{Fin}} 
\newcommand{\Ctbl}{\textrm{Ctbl}} 
\newcommand{\cc}{\mf c}

\newcommand{\I}{\mathcal I}
\newcommand{\J}{\mathcal J}
\newcommand{\h}{\mathcal H}
\newcommand{\Null}{\mathcal N}
\newcommand{\M}{\mathcal M}

\title[Haar-smallest sets]{Haar-smallest sets}

\author{Adam Kwela}
\address{Institute of Mathematics, Faculty of Mathematics, Physics and Informatics, University of Gda\'{n}sk, ul.~Wita Stwosza 57, 80-308 Gda\'{n}sk, Poland}
\email{adam.kwela@ug.edu.pl}

\begin{document}
\begin{abstract}
In this paper we are interested in the following notions of smallness: a subset $A$ of an abelian Polish group $X$ is called Haar-countable/Haar-finite/Haar-$n$ if there are a Borel hull $B\supseteq A$ and a copy $C$ of $2^\omega$ in $X$ such that $(C+x)\cap B$ is countable/finite/of cardinality at most $n$, for all $x\in X$. 

Recently, Banakh et al. have unified the notions of Haar-null and Haar-meager sets by introducing Haar-$\mathcal{I}$ sets, where $\mathcal{I}$ is a collection of subsets of $2^\omega$. It turns out that if $\mathcal{I}$ is the $\sigma$-ideal of countable sets, the ideal of finite sets or the collection of sets of cardinality at most $n$, then we get the above notions. Moreover, those notions have been studied independently by Zakrzewski (under a different name -- perfectly $\kappa$-small sets).

We study basic properties of the corresponding families of small sets, give suitable examples distinguishing them (in all groups of the form $\mathbb{R}\times X$, where $X$ is an abelian Polish group) and study $\sigma$-ideals generated by closed members of the considered families. In particular, we show that Haar-finite sets do not form an ideal. Moreover, we answer some questions concerning null-finite sets asked by Banakh and Jab\l o\'{n}ska and pose several open problems.
\end{abstract}

\maketitle

\section{Introduction}

We follow standard set-theoretic and topological notation and terminology (see \cite{Kechris}). In particular, by $\omega$ we denote the set $\{0,1,\ldots\}$, identify each $n\in\omega$ with the set $\{0,1,\ldots,n-1\}$ and denote by $|X|$ the cardinality of a set $X$. By a Cantor set we mean a homeomorphic copy of the space $2^\omega$.

All Polish groups studied in the paper are assumed to be uncountable.

In this paper we are interested in the following notions of Haar-smallness.

\begin{df}
Let $A$ be a subset of an abelian Polish group $X$.
\begin{itemize} 
	\item $A$ is Haar-countable ($A\in\h\Ctbl$) if there are a Borel hull $B\supseteq A$ and a copy $C$ of $2^\omega$ in $X$ such that $(C+x)\cap B$ is countable for all $x\in X$;
	\item $A$ is Haar-finite ($A\in\h\Fin$) if there are a Borel hull $B\supseteq A$ and a copy $C$ of $2^\omega$ in $X$ such that $(C+x)\cap B$ is finite for all $x\in X$;
	\item $A$ is Haar-$n$ ($A\in\h n$), for $n\in\omega\setminus\{0\}$, if there are a Borel hull $B\supseteq A$ and a copy $C$ of $2^\omega$ in $X$ such that $|(C+x)\cap B|\leq n$ for all $x\in X$.
\end{itemize}
\end{df}

Clearly,
$$\text{Haar-}n\ \Longrightarrow\ \text{Haar-}(n+1)\ \Longrightarrow\ \text{Haar-finite}\ \Longrightarrow\ \text{Haar-countable}$$
for any $n\in\omega\setminus\{0\}$. 

By removing the requirement about Borel hulls from the above definition, we would obtain the notions of \emph{perfectly $\kappa$-small sets} studied by Zakrzewski in \cite{Zakrzewski} in the context of ccc $\sigma$-ideals. The particular case of Haar-$1$ sets has been investigated by Balcerzak in \cite{Balcerzak}, where he introduced the so-called property (D) of an invariant $\sigma$-ideal $\I$, which says that there is a $\sigma$-compact Haar-$1$ set not belonging to $\I$. Moreover, Banakh, Lyaskovska and Repov\v s in \cite{BLR} considered packing index of a set, which is closely connected to Haar-$1$ sets (namely, a Borel set is Haar-$1$ if and only if its packing index is uncountable). In turn, Haar-countable sets have been studied by Darji and Keleti in \cite{DK}. Let us point out a connection of Haar-countable sets with covering the whole space by less than continuum translates of its subset: less than continuum translates of a Haar-countable set cannot cover the whole space (cf. \cite[Lemma 4.1]{Zakrzewski}). This fact links Haar-countable sets with the works of Dobrowolski, Elekes, Marciszewski, Miller, Stepr\=ans and T\'oth (see \cite{DM}, \cite{Elekes}, \cite{ElekesToth} and \cite{MS}).

The following discussion will explain the choice of names in the above definition.

Recall that a collection of subsets of a set $X$ is called:
\begin{itemize} 
	\item a \emph{semi-ideal} on $X$, if it is closed under taking subsets; 
	\item an \emph{ideal} on $X$, if it is closed under taking subsets and finite unions; 
	\item a \emph{$\sigma$-ideal} on $X$, if it is closed under taking subsets and countable unions.
\end{itemize}

In the absence of Haar measure in Polish groups which are not locally compact, in \cite{Christensen} Christensen introduced the notion of Haar-null sets. A big advantage of this concept is that in a locally compact group a set is Haar-null if and only if it is of Haar measure zero and at the same time it can be used in a significantly larger class of groups. In \cite{HSY} Hunt, Sauer and Yorke, unaware of Christensen's paper, reintroduced the notion of Haar-null sets in the context of dynamical systems  (in their paper Haar-null sets are called shy sets and their complements are called prevalent sets). A topological analogue of Haar-null sets was defined by Darji in \cite{Darji}. Similarly as before, in the locally compact case this notion is equivalent to the notion of meager sets. However, Haar-meager sets enable us to capture in a single notion of smallness both group and topological structures. It is known that Haar-null and Haar-meager sets form $\sigma$-ideals (cf. \cite{Christensen}, \cite{Darji} and \cite{HSY}).  

Recently, the notions of Haar-null sets and Haar-meager sets were unified by Banakh et al. in \cite{1} by introducing the concept of Haar-small sets. For a semi-ideal $\I$ on the Cantor cube $2^\omega$, we say that a subset $A$ of an abelian Polish group $X$ is \emph{Haar-$\I$} ($A\in\h\I$) if there are a Borel hull $B\supseteq A$ and a continuous map $f\colon 2^\omega\to X$ with $f^{-1}[B+x]\in\I$ for all $x\in X$. Obviously, the collection of Haar-$\I$ subsets of an abelian Polish group is a semi-ideal. It turns out that if $\I$ is the $\sigma$-ideal $\mathcal{N}$ of subsets of $2^\omega$ of Lebesgue's measure zero, then we obtain Haar-null sets (cf. \cite{1}). The same holds for the $\sigma$-ideal $\mathcal{M}$ of meager subsets of $2^\omega$ and Haar-meager sets (cf. \cite{1}).

In this paper we will restrict ourselves only to the following semi-ideals:
\begin{itemize} 
	\item the $\sigma$-ideal $[2^\omega]^{\leq\omega}$ of at most countable subsets of $2^\omega$;
	\item the ideal $[2^\omega]^{<\omega}$ of finite subsets of $2^\omega$;
	\item the semi-ideals $[2^\omega]^{\leq n}$ of subsets of $2^\omega$ of cardinality at most $n$, for $n\in\omega$.
\end{itemize}

Clearly, each Haar-countable set is Haar-null and Haar-meager. As a consequence, a countable union of Haar-countable sets cannot be the whole space (see also \cite[Corollary 3.7]{Zakrzewski} where this fact is pointed out in the case of the group $2^\omega$). 

In the mentioned cases we have a nice characterization, which connects notions introduced at the beginning of this paper with the above considerations (by showing that $\h[2^\omega]^{\leq\omega}=\h\Ctbl$, $\h[2^\omega]^{<\omega}=\h\Fin$ and $\h[2^\omega]^{\leq n}=\h n$). Equivalence of items $(b)$ and $(c)$ in the following proposition is implicit in \cite[Lemma 3.1]{Zakrzewski}. 

In the proofs we will often use item $(c)$ instead of the above definition of Haar-$\I$ sets without any reference to the following proposition.

\begin{prop}
Let $\alpha$ stand for one of the symbols $\leq\omega$, $<\omega$ or $\leq n$ for some $n\in\omega$. Then the following are equivalent for any Borel subset $B$ of an abelian Polish group $X$:
\begin{itemize}
	\item[(a)] $B\in\h[2^\omega]^\alpha$;
	\item[(b)] there is a Cantor set $D\subseteq X$ with $(D-x)\cap B\in[X]^\alpha$ for each $x\in X$;
	\item[(c)] there is a Cantor set $D\subseteq X$ such that for any $S\subseteq D$ with $S\notin[X]^\alpha$ we have $\bigcap_{s\in S}(B-s)=\emptyset$.
\end{itemize}
\end{prop}

\begin{proof}
Firstly, we show that $(a)$ implies $(b)$. Let $f\colon 2^\omega\to X$ be continuous and such that $f^{-1}[B+x]\in [2^\omega]^\alpha$ for all $x\in X$. Without loss of generality we can assume that $f[2^\omega]$ is uncountable (otherwise $f$ could be a witness only for $B=\emptyset$ and this case is easy). As $f[2^\omega]$ is also compact, we can find a Cantor set $D\subseteq f[2^\omega]$ (see \cite{Kechris}). Now it suffices to observe that 
$$|f^{-1}[B+x]|\geq|f[2^\omega]\cap(B+x)|\geq|D\cap(B+x)|=|(D-x)\cap B|$$
for every $x\in X$, so $f^{-1}[B+x]\in [2^\omega]^\alpha$ implies $(D-x)\cap B\in[X]^\alpha$.

To see that  $(b)$ implies $(a)$ it suffices to take a homeomorphism $f\colon 2^\omega\to D$.

Finally, the equivalence of $(b)$ and $(c)$ is a consequence of the fact that both assertions can be written as: there is a Cantor set $D\subseteq X$ such that for every $x\in X$ and every $S\subseteq D$ with $S\notin[X]^\alpha$ there is such $s\in S$ that $s-x\notin B$. 
\end{proof}

Hunt, Sauer and Yorke introduced Haar-null sets in order to state mathematically precisely sentences saying that some property holds for almost all elements of a given infinite-dimensional linear space such as $C([0,1])$ or $L^1([0,1])$. This was supposed to be a counterpart of sets of measure zero in such spaces. As $\I\subseteq\J$ implies $\h\I\subseteq\h\J$, the notion of Haar-small sets enables us to develop a whole hierarchy of systems of small subsets of abelian Polish groups. Next two propositions illustrate the above by distinguishing two of the rare properties which motivated Hunt, Sauer and Yorke for introducing Haar-null sets. Note that the set $\mathcal{A}$ of functions with $\int_0^1 f(x)dx=0$ is Haar-null in $L^1[0,1]$ by \cite{HSY} and the set $\mathcal{SD}$ of somewhere differentiable functions (i.e., functions which have a derivative at some point) is Haar-null in $C[0,1]$ by \cite{Hunt}.

\begin{prop} 
The set $\mathcal{A}$ of functions with $\int_0^1 f(x)dx=0$ is Haar-$1$ in $L^1[0,1]$ (as well as in $C[0,1]$).
\end{prop}

\begin{proof}
It is easy to check that $\mathcal{A}$ is Borel.

Let $C\subseteq\R$ denote the standard ternary Cantor set. For each $c\in C$ define $\phi(c)\colon [0,1]\to\R$ by $\phi(c)(x)=c$ for all $x\in [0,1]$. It is easy to check that $\phi\colon C\to L^1[0,1]$ is a continuous map. 

As $C$ and $2^\omega$ are homeomorphic, to prove the first assertion it suffices to show that for every $h\in L^1 [0,1]$ the set $\phi^{-1}[\mathcal{A}+h]$ has cardinality at most $1$ or, in other words, for every $h\in L^1 [0,1]$ we have $\phi(c)-h\in\mathcal{A}$ for at most one $c\in C$. Fix any $h\in L^1[0,1]$ and suppose that $\int_0^1 (\phi(c)-h)(x)dx=0$ for some $c\in C$. Then for each $c'\in C$, $c'\neq c$, we have $\int_0^1 (\phi(c')-h)(x)dx=c'-c\neq 0$. Hence, $\phi$ is the desired map.

As all continuous functions are integrable, in the case of $C[0,1]$ the proof is exactly the same.
\end{proof}

\begin{prop}
\label{ND}
The set $\mathcal{SD}$ of somewhere differentiable functions is not Haar-finite in $C[0,1]$.
\end{prop}

\begin{proof}
Let $\phi\colon 2^\omega\to C[0,1]$ be a continuous map. We need to show that $\phi^{-1}[\mathcal{SD}+h]$ is infinite for some $h\in C[0,1]$.

If $\phi[2^\omega]$ is finite, then let $h\in C[0,1]$ be such that $\phi^{-1}[\{h\}]$ is infinite and observe that $h\in\mathcal{SD}+h$. Hence, $\phi^{-1}[\mathcal{SD}+h]$ is infinite as well. 

Assume now that $\phi[2^\omega]$ is infinite and take any injective convergent sequence $(f_n)\in\phi[2^\omega]$. Denote $f=\lim_{n}f_n$. Fix a sequence of closed intervals $(I_n)$ such that $I_n\subseteq (\frac{1}{n+2},\frac{1}{n+1})$ for each $n\in\omega$. Let $J_n=[\max I_{n+1},\min I_n]$ and $g_n\colon J_n\to\R$ be the linear function given by $g_n(\max I_{n+1})=0$ and $g_n(\min I_{n})=(f_n-f_{n+1})(\min I_n)$. Define $h\colon[0,1]\to\mathbb{R}$ by:
\begin{itemize}
	\item $h\upharpoonright [\max I_0,1]=f_0(\max I_0)$;
	\item $h\upharpoonright I_n=f_n\upharpoonright I_n$ for each $n\in\omega$;
	\item $h\upharpoonright J_n=(f_{n+1}+g_n)\upharpoonright J_n$ for each $n\in\omega$;
	\item $h(0)=f(0)$.
\end{itemize}
Note that the function $h$ is continuous (continuity in $0$ follows from $\lim_{n}f_n=f$ and $\lim_n \sup_{x\in J_n}|g_n(x)|=0$) and $f_n\in\mathcal{SD}+h$ for each $n\in\omega$ (as $(f_n-h)\upharpoonright I_n$ is constant). Since $(f_n)$ is injective, we conclude that $\phi^{-1}[\mathcal{SD}+h]$ is infinite.
\end{proof}

Actually, in the forthcoming paper \cite{ND} we use Haar-smallest sets in further studies of differentiability of continuous functions. For instance, we prove that $\mathcal{SD}$ is even not Haar-countable.

Although the concept of Haar-small sets is very recent, its variation described below had already been applied to solve an old problem of Baron and Ger from \cite{Ger}. We say that a subset $A$ of an abelian Polish group $X$ is \emph{null-finite} if there is a convergent sequence $(x_n)_{n\in\omega}\subseteq X$ such that $\{n\in\omega:\ x_n-x\in A\}$ is finite for all $x\in X$. This notion is very closely related to Haar-finite sets: a witness for a null-finite set is compact (i.e., the compact set $\{x_n:\ n\in\omega\}\cup\{\lim_n x_n\}$ has an analogous property to that of witnesses of Haar-finite sets), but it does not have to be uncountable as in the case of Haar-finite sets. Clearly, each Haar-finite set is null-finite. 

Null-finite sets were introduced by Banakh and Jab\l o\' nska in order to solve the mentioned problem of Baron and Ger concerning functional equations: a mid-point convex function $f\colon G\to \mathbb{R}$ defined on an open convex subset $G$ of a complete linear metric space $X$ is continuous if it is upper bounded on a Borel subset $B\subseteq G$ which is not Haar null or not Haar meager in $X$. 

The main aim of this paper is to develop a hierarchy of Haar-small sets. This requires constructing examples distinguishing $\h\I$ for various semi-ideals $\I$. This paper studies relationships between families $\h n$, $\h\Fin$, $\h\Ctbl$ and $\h\Null\cap\h\M$. We give such examples for the considered semi-ideals on the real line. However, in Section $2$ we show that our results can be applied to a wider class of abelian Polish groups. 

In Section $3$ we show that all countable sets are Haar-$1$ and construct an uncountable compact Haar-$1$ set. In Section $4$ we distinguish $\h 2$ from $\h 1$ by showing that the standard ternary Cantor set $C$ is Haar-$2$ but not Haar-$1$. Also, we conclude that a closed Haar-$2$ set does not have to be a countable union of closed Haar-$1$ sets. Section $5$ is devoted to proving that $\h (n+1)\setminus\h n$ is nonempty.

In the longest and most technical Section \ref{Haar-finite}, answering a question posed by Swaczyna during his talk on XLI Summer Symposium in Real Analysis (Wooster, 2017), we show that the family of all Haar-finite subsets of $\mathbb{R}$ is not an ideal. This means that $\I$ being an ideal does not determine whether $\h\I$ is an ideal. What is more, we get an example of a compact Haar-finite set which is not Haar-$n$ for any $n$.

In the last section we construct a compact Haar-countable set which is not Haar-finite and a compact null and meager set which is neither Haar-countable nor a countable union of closed Haar-countable sets. Thus, the $\sigma$-ideal generated by closed Haar-countable sets is a proper subset of $\mathcal{N}\cap\mathcal{M}$ (which is the same as $\h\mathcal{N}\cap\h\mathcal{M}$, since we are on $\R$).

All our results concerning Haar-finite sets can be adapted also for null-finite sets. Consequently, we show that the family of all null-finite subsets of $\mathbb{R}$ is not an ideal. This answers a question posed in the first version of \cite{2} and asked by Banakh during his talk at the conference Frontiers of Selection Principles (Warsaw, 2017). Moreover, in Section $7$ we construct a compact null and meager set which is not a countable union of closed null-finite sets. This is a partial solution to another problem from \cite{2}. Actually, we do not know whether there is a Borel null-finite subset of $\R$ which is not Haar-finite.

\section{Preliminaries} 

Let us point out a simple observation.

\begin{remark}
Observe that $\h[2^\omega]^{\leq\omega}=\h[2^\omega]^{<\cc}$. Indeed, a set $A$ is in $\h[2^\omega]^{<\cc}$ provided that there are a Borel $B\supseteq A$ and a continuous map $\phi\colon 2^\omega\to X$ with $|\phi[2^\omega]\cap (B+x)|<\cc$ for each $x\in X$. However, in this case the set $\phi[2^\omega]\cap (B+x)$ is countable as a Borel set (cf. \cite{Kechris}).
\end{remark}

If $\mathcal{A}\subseteq\mathcal{P}(X)$, then $\sigma\mathcal{A}$ is the $\sigma$-ideal generated by $\mathcal{A}$, i.e., the family of sets which can be covered by a countable union of elements of $\mathcal{A}$. Moreover, if $X$ is a topological space, then $\overline{\mathcal{A}}$ denotes the family of closed sets belonging to $\mathcal{A}$.

The following series of results collects some known facts about Haar-smallest sets. Some of them will be often used in our considerations.

\begin{thm}[{\cite[Proposition 10.2]{1}}]
For every semi-ideal $\I$ on $2^\omega$ the family $\h\I$ of all Haar-$\I$ sets in any abelian Polish group is translation-invariant and preserved under topological group isomorphisms (i.e., $\h\I=\{f[A]:\ A\in\h\I\}$ for each topological group isomorphism $f$). In particular, a linear copy of a Haar-$\I$ subset of $\R$ is still Haar-$\I$. 
\end{thm}

\begin{thm}[Essentially {\cite[Proposition 12.6 and Corollary 13.4(3)]{1}}]
\label{sumyHn}
For a fixed abelian Polish group $X$, if $A,B\subseteq X$ are such that $A\in\overline{\h n}$ and $B\in\overline{\h m}$ for some $n,m\in\omega\setminus\{0\}$, then $A\cup B\in\overline{\h (n+m)}$. Moreover, we have $\sigma\left(\bigcup_{n\in\omega}\overline{\h n}\right)\subseteq\h\Ctbl$.
\end{thm}

\begin{proof}
By \cite[Corollary 13.4(3)]{1}, for an arbitrary $k\in\omega\setminus\{0\}$, every $\sigma$-compact Haar-$k$ subset of $X$ is generically Haar-$k$, i.e., the set of functions $f\colon 2^\omega\to X$ witnessing it is comeager in the space $C(2^\omega,X)$. 

For the first assertion, observe that the sets of functions witnessing that $A\in\h n$ as well as that $B\in\h m$ are comeager in $C(2^\omega,X)$. Take any function $f\colon 2^\omega\to X$ from the intersection of those two sets. Then 
$$|f^{-1}[(A\cup B)+x]|=|f^{-1}[A+x]\cup f^{-1}[B+x]|\leq |f^{-1}[A+x]|+|f^{-1}[B+x]|\leq n+m.$$
Thus, $A\cup B\in\overline{\h (n+m)}$.

For the second assertion, fix any $(A_n)_n\subseteq \bigcup_{n\in\omega}\overline{\h n}$. Then each $A_n$ is generically Haar-$k$ for some $k$, hence generically Haar-countable. By \cite[Proposition 12.6]{1}, the family of all generically Haar-countable sets is an invariant $\sigma$-ideal. In particular, $\bigcup_n A_n$ is generically Haar-countable.
\end{proof}

\begin{thm}[{\cite[Theorem 2.5]{DK}}]
\label{PackingDim}
Every compact subset of $\mathbb{R}$ with packing dimension less than $1$ is Haar-finite.
\end{thm}

The following series of results shows that all results of this paper can be applied to abelian Polish groups of the form $\R\times X$, where $X$ is an arbitrary closed abelian group. In particular, this class contains spaces $C(K)$, where $K$ is a compact metrizable space. 

\begin{prop}
Let $X$ and $Y$ be two nonempty abelian Polish groups. Fix a semi-ideal $\I$ on $2^\omega$ and $A\subseteq X$.
\begin{itemize}
	\item[(i)] If $A\in\h\I$ ($\overline{\h\I}$), then $A\times Y\in\h\I$ ($\overline{\h\I}$).
	\item[(ii)] If $A\notin\h\I$ ($\overline{\h\I}$), then $A\times Y\notin\h\I$ ($\overline{\h\I}$).
	\item[(iii)] If $A\in\sigma\overline{\h\I}$, then $A\times Y\in\sigma\overline{\h\I}$.
\end{itemize}
\end{prop}

\begin{proof}
Note that $A$ is closed in $X$ if and only if $A\times Y$ is closed in $X\times Y$. Therefore, we only need to prove the assertions (i) and (ii) for $\h\I$.

{\bf (i): }Let $\phi\colon 2^\omega\to X$ be a witness for $A\in\h\I$. Fix any $y\in Y$ and define $\psi\colon 2^\omega\to X\times Y$ by $\psi(\alpha)=(\phi(\alpha),y)$ for each $\alpha\in 2^\omega$. Then $\psi$ is continuous and for each $(r,r')\in X\times Y$ we have $\psi^{-1}[(A\times Y)+(r,r')]=\phi^{-1}[A+r]\in\I$. 

{\bf (ii): }Let $\psi\colon 2^\omega\to X\times Y$ be continuous. Then $\phi\colon 2^\omega\to X$ given by $\phi(\alpha)=\pi_{X}(\psi(\alpha))$, where $\pi_X\colon X\times Y\to X$ is the projection map, is continuous as well. Hence, there is $r\in X$ with $\phi^{-1}[A+r]\notin\I$. Then we have $\psi^{-1}[(A\times Y)+(r,r')]\supseteq\phi^{-1}[A+r]\notin\I$ for each $r'\in Y$.

{\bf (iii): }This is an immediate consequence of item (i).
\end{proof}

\begin{prop}
Fix two nonempty abelian Polish groups $X$ and $Y$. Let $\I$ stand for $[2^\omega]^{\leq\omega}$, $[2^\omega]^{<\omega}$ or $[2^\omega]^{\leq n}$ for some $n\in\omega$. Suppose that $A\subseteq X$ is closed and such that $A\cap U\notin\h\I$ for each open $U\subseteq X$ with $A\cap U\neq\emptyset$. Then $A\times Y\notin\sigma\overline{\h\I}$.
\end{prop}

\begin{proof}
Suppose to the contrary that $A\times Y=\bigcup_{n\in\omega}Z_n$ for some sequence of closed Haar-$\I$ sets $(Z_n)_n$. As $A\times Y$ is closed, by the Baire's category theorem, we can find $m\in\omega$ such that $Z_m$ has nonempty interior in the space $A\times Y$ (with the subspace topology). Hence, there are open sets $U\subseteq X$ and $V\subseteq Y$ with $\emptyset\neq(A\cap U)\times V\subseteq Z_m$. We will show that $Z_m\notin\h\I$ which is a contradiction.

Let $\psi\colon 2^\omega\to X\times Y$ be continuous. Consider the continuous function $\psi_Y\colon 2^\omega\to Y$ given by $\psi_Y(\alpha)=\pi_{Y}(\psi(\alpha))$, where $\pi_Y$ is the projection map. Find any $y\in Y$ with $\psi_Y^{-1}[V+y]\neq\emptyset$. Then $\psi_Y^{-1}[V+y]$, as an open set, must contain a nonempty clopen set $B\subseteq 2^\omega$. 

Note that $B$ is homeomorphic to $2^\omega$ and denote this homeomorphism by $\sigma\colon 2^\omega\to B$. Consider $\phi\colon 2^\omega\to X$ given by $\phi(\alpha)=\pi_X(\psi(\sigma(\alpha)))$, where $\pi_X$ is the projection map. Note that $\phi$ is continuous and $A\cap U\neq\emptyset$ (otherwise $(A\cap U)\times V$ would be empty). Thus, there is $x\in X$ with $\phi^{-1}[(A\cap U)+x]\notin\I$. Then we have $\psi^{-1}[Z_m+(x,y)]\notin\I$ as 
$$|\psi^{-1}[((A\cap U)\times V)+(x,y)]|\geq|B\cap\psi^{-1}[\pi_X^{-1}[(A\cap U)+x]]|=|\phi^{-1}[(A\cap U)+x]|.$$
\end{proof}

\begin{remark}
\label{sigma}
Note that for a closed set $A\subseteq\R$ the assumption $A\cap I\notin\h\I$, for each open interval $I\subseteq\R$ with $A\cap I\neq\emptyset$, implies $A\notin\sigma\overline{\h\I}$. This is a simple consequence of the Baire's category theorem (actually, in the case of the considered semi-ideals it also follows from the previous proposition -- it suffices to take any $Y$ of cardinality $1$). We will often refer to this basic observation in our further considerations. 
\end{remark}

\section{Haar-$1$ sets}

In this section we study Haar-$1$ sets -- we show that all countable subsets are Haar-$1$ and give an example of an uncountable Haar-$1$ subset of the real line.

\begin{prop}
\label{przeliczalne}
Let $X$ be an abelian Polish group. All countable subsets of $X$ are Haar-$1$, i.e., $[X]^{\leq\omega}\subseteq\h 1$.
\end{prop}

\begin{proof}
We will use \cite[Corollary 2.3]{Kubis} stating that for a symmetric $\mathtt{F_\sigma}$ relation $R$ on $X$ with an $R$-independent set of size $\cc$ we can always find a Cantor $R$-independent set.

Let $\emptyset\neq A\subseteq X$ be countable. Define a relation $R\subseteq X\times X$ by: 
$$(x,y)\in R \Leftrightarrow x-y\in A-A$$
for all $x,y\in X$. Note that it is reflexive and symmetric. Moreover, as $A-A$ is countable, it is easy to check that $R$ is $\mathtt{F_\sigma}$. 

We will show that there is an $R$-independent set (i.e., $(x,y)\notin R$ for each two distinct elements $x$ and $y$ of that set) of cardinality $\cc$. Indeed, by transfinite induction one can construct $x_\alpha$, for $\alpha<\cc$, satisfying $x_\alpha\notin\bigcup_{\beta<\alpha}\{y\in X:\ (y,x_\beta)\in R\}$ (as $\{y\in X:\ (y,x_\beta)\in R\}$ is countable for each $\beta<\alpha$). Then $X=\{x_\alpha:\ \alpha<\cc\}$ is the required set. 

By \cite[Corollary 2.3]{Kubis}, there is a Cantor set $D\subseteq X$ with $(x,y)\notin R$ for all $x,y\in D$, $x\neq y$. We claim that $D$ witnesses $A\in\h 1$. Suppose to the contrary that $|(D-r)\cap A|>1$ for some $r\in X$. Hence, there are $x,y\in D$, $x\neq y$, with $x-r,y-r\in A$. But then $(x,y)\in R$, since $(x-r)-(y-r)=x-y\in A-A$. A contradiction.
\end{proof}

Now we will give an example of a compact uncountable Haar-$1$ set. This will follow from the following more general result (see also \cite[Section 2]{BLR} and \cite[Proposition 3.1]{Balcerzak}).

\begin{prop}[{\cite[Remark 3.3]{Zakrzewski}}]
\label{null-1}
Let $A$ be a compact subset of an abelian Polish group $X$. The following are equivalent:
\begin{enumerate}
\item $A$ is Haar-$1$;
\item $A-A$ does not contain a neighborhood of the neutral element of $X$;
\item $A-A$ has empty interior.
\end{enumerate}
\end{prop}

\begin{cor}
\label{Cm}
Let $m\in\omega\setminus 4$. The set $A=\left\{\sum_{n\in\omega}\frac{i_n}{m^{n+1}}:\ \forall_{n\in\omega}\ i_n\in\{0,m-1\}\right\}$ is Haar-$1$.
\end{cor}

\begin{proof}
Note that $d_n=\frac{m-1}{2\cdot m^{n+1}}$ is a converging to $0$ sequence with $(A-d_n)\cap A=\emptyset$ for all $n$. Thus, $d_n\in A-A$ for all $n$. It follows that $A-A$ does not contain a neighborhood of $0$. By item (2) of the previous proposition we get that $A$ is Haar-$1$.
\end{proof}

\section{Haar-$2$ sets}

In this section we distinguish $\overline{\h 2}$ from $\h 1$ and $\sigma\overline{\h 1}$.

Notice that the next theorem, besides providing a suitable example (the ternary Cantor set), explains why we had to assume $m>3$ in Corollary \ref{Cm}. Moreover, it strengthens \cite[Remark 3.10]{Zakrzewski}, where it is observed (with the use of Theorem \ref{PackingDim}) that the ternary Cantor set is Haar-countable but not Haar-$1$. Observe also that a required example (of a compact Haar-$2$ set which is not Haar-$1$) in the case of the group $2^\omega$ has been constructed by Zakrzewski in \cite[Theorem 3.8]{Zakrzewski}.

\begin{thm}
\label{one-third}
The ternary Cantor set $C$ is Haar-$2$ but not Haar-$1$, i.e., $C\in\overline{\h 2}\setminus\h 1$.
\end{thm}

\begin{proof}
First we show that $C$ is not Haar-$1$. This is easy using Proposition \ref{null-1}, since by a well-known theorem of Steinhaus, $C-C=[-1,1]$.

Now we prove that $C$ is Haar-$2$. We start with some useful observations. Firstly, note that $C\cap(C-(\frac{1}{3^1}+\frac{1}{3^2}))\cap(C-\frac{2}{3^2})=\emptyset$ (as $C\cap(C-(\frac{1}{3^1}+\frac{1}{3^2}))\subseteq[\frac{2}{9},\frac{3}{9}]$ and $C\cap(C-\frac{2}{3^2})\cap[\frac{2}{9},\frac{3}{9}]$ is empty -- see Figure 1). Actually, we can generalize the above observation a little bit: we have $C\cap (C-((\frac{1}{3^1}+\frac{1}{3^2})+a'))\cap (C-(\frac{2}{3^2}+b'))=\emptyset$ for any $a',b'\in[0,\frac{1}{9})$ (as $\text{dist}(C\cap (C-(\frac{1}{3^1}+\frac{1}{3^2})),(C-\frac{2}{3^2}))=\frac{1}{3^2}$). 

We can further generalize our observation: fix two sequences $\bar{a}=(a_i)_i,\bar{b}=(b_i)_i\in 3^k$ and denote $a=\sum_{i\in k}\frac{a_i}{3^{i+1}}$, $b=\sum_{i\in k}\frac{b_i}{3^{i+1}}$. Let $\alpha$ and $\beta$ stand for the number of $1$'s in $\bar{a}$ and $\bar{b}$, respectively. Note first that $C\cap (C-a)$ is either finite (if $\alpha$ is odd) or equal to $C\cap I$ where $I$ is some finite union of intervals of the form $[\frac{l}{3^k},\frac{l+1}{3^k}]$, $l<3^k$ (if $\alpha$ is even). Consider first the case that both $\alpha$ and $\beta$ are even. Then, similarly as in the beginning of the previous paragraph, for $x'=\frac{1}{3^{k+1}}+\frac{1}{3^{k+2}}$ and $y'=\frac{2}{3^{k+2}}$ we have $C\cap (C-(a+x'+a'))\cap (C-(b+y'+b'))=\emptyset$ for any $a',b'\in[0,\frac{1}{3^{k+2}})$. 

\begin{figure}[H]
\begin{center}
\includegraphics[scale=0.24]{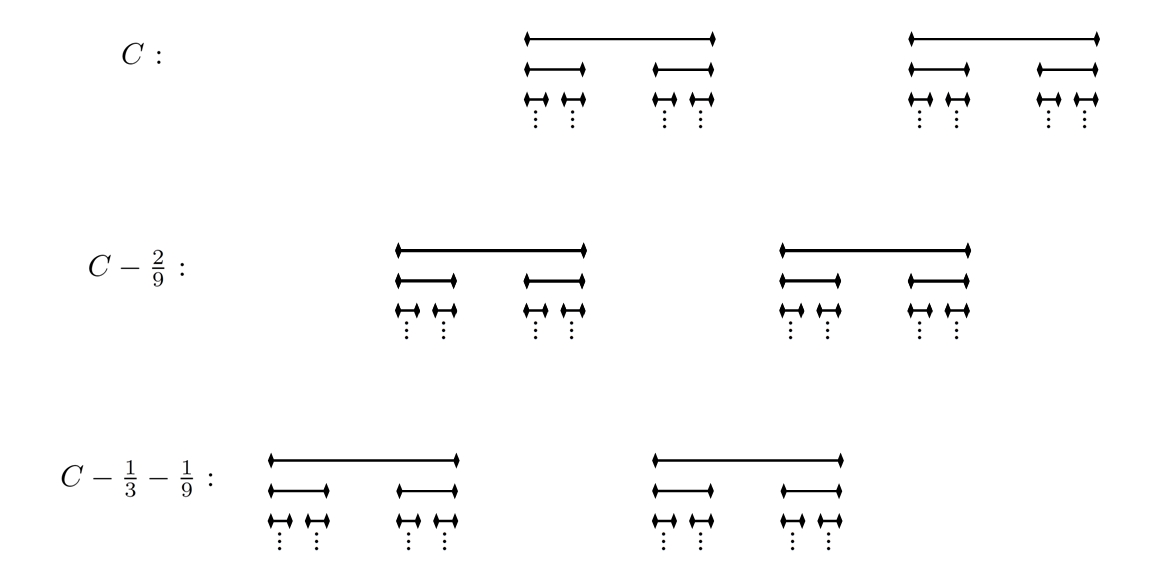}
\end{center}
\caption{}
\end{figure}

Using the above observation we can see that in the general case (when $\alpha$ and $\beta$ are arbitrary), by putting 
$$\bar{x}=(x_i)_{i\in 11}=(1,1,1,1,1,0,1,1,1,1,1)$$ $$\bar{y}=(y_i)_{i\in 11}=(0,2,0,0,2,1,0,2,0,0,2)$$
and denoting $x=\sum_{i\in 11}\frac{x_i}{3^{k+1+i}}$ and $y=\sum_{i\in 11}\frac{y_i}{3^{k+1+i}}$, we have $C\cap (C-(a+x+a'))\cap (C-(b+y+b'))=\emptyset$ for any $a',b'\in[0,\frac{1}{3^{k+11}})$. Indeed, the first two elements of $\bar{x}$ and $\bar{y}$ guarantee us that this intersection is empty when both $\alpha$ and $\beta$ are even. If $\alpha$ is odd and $\beta$ is even, then the third elements of $\bar{x}$ and $\bar{y}$ lead us to the case that both $\alpha+|\{i\in 3:\ x_i=1\}|$ (i.e., the number of $1$'s in $\bar{a}^\frown (\bar{x}\upharpoonright 3)$) and $\beta+|\{i\in 3:\ y_i=1\}|$ (i.e., the number of $1$'s in $\bar{b}^\frown (\bar{y}\upharpoonright 3)$) are even and using the fourth and fifth elements of $\bar{x}$ and $\bar{y}$ we see that the considered intersection is empty. The two remaining cases are similar -- if both $\alpha$ and $\beta$ are odd, then the numbers of $1$'s in $\bar{a}^\frown (\bar{x}\upharpoonright 6)$ and $\bar{b}^\frown (\bar{y}\upharpoonright 6)$ are even, and if $\alpha$ is even and $\beta$ is odd, then the numbers of $1$'s in $\bar{a}^\frown (\bar{x}\upharpoonright 9)$ and $\bar{b}^\frown (\bar{y}\upharpoonright 9)$ are even).

We are ready to construct a Cantor set $D$ witnessing that $C$ is Haar-$2$. Define $k_0=0$, $k_1=1$ and $k_{n}=k_{n-1}+12\left(\genfrac{}{}{0pt}{}{2^n}{3}\right)$ for $n>1$ (where $\left(\genfrac{}{}{0pt}{}{2^n}{3}\right)$ denotes the number of $3$-combinations of the set $2^n$). We inductively pick finite sequences $d^s=(d^{s}_i)_i\in 3^{k_n-k_{n-1}}$ for every $s\in 2^n$ and $n\geq 1$. At the end $D$ will consist of all points of the form $\sum_{n\geq 1}\bar{d}_{x|n}$ for some $x\in 2^\omega$ where $\bar{d}_{s}=\sum_{i=1}^{k_n-k_{n-1}}\frac{d^s_{i-1}}{3^{k_{n-1}+i}}$ for $s\in 2^n$, $n\geq 1$. We start with $d^{(0)}=(0)$ and $d^{(1)}=(1)$ (here $(0)$ and $(1)$ denote the sequence of length one consisting of the element $0$ and $1$, respectively). In the $n$th step define $d^s$ for each $s\in 2^n$ in such a way that:
\begin{itemize}
	\item $d^s$ is a concatenation of $\left(\genfrac{}{}{0pt}{}{2^n}{3}\right)$ sequences each of which is equal to $\bar{x}^\frown(1)$, $\bar{y}^\frown(1)$ or $\bar{0}_{12}$, i.e., the sequence consisting of twelve $0$'s;
	\item for each pairwise distinct $s,s',s''\in 2^n$ with $s<_{lex}s'<_{lex}s''$, where $<_{lex}$ denotes the lexicographic order on $2^{<\omega}$ (note that there are exactly $\left(\genfrac{}{}{0pt}{}{2^n}{3}\right)$ ways to choose $s,s',s''$) we can find $F=\{12p,\ldots,12p+11\}$ for some $p<\left(\genfrac{}{}{0pt}{}{2^n}{3}\right)$ such that $(d^s_i)_{i\in F}=\bar{0}_{12}$, $(d^{s'}_i)_{i\in F}=\bar{x}^\frown(1)$ and $(d^{s''}_i)_{i\in F}=\bar{y}^\frown(1)$.
\end{itemize}
By the construction, given any $d,d',d''\in D$ with $d<d'<d''$, we can find $12$ consecutive positions $\{q,q+1,\ldots,q+11\}$ such that on these positions, the $3$-ary expansion $(d_i)\in 3^\omega$ of $d$ is given by $\bar{0}_{12}$, the $3$-ary expansion $(d'_i)\in 3^\omega$ of $d'$ is given by $\bar{x}^\frown(1)$ and the $3$-ary expansion $(d''_i)\in 3^\omega$ of $d''$ is given by $\bar{y}^\frown(1)$. Then on the first $11$ of those positions $\{q,q+1,\ldots,q+10\}$ the $3$-ary expansions of $d'-d$ and $d''-d$ are given by $\bar{x}$ and $\bar{y}$, respectively. Indeed, 
$$d'-d=\sum_{i\in\omega}\frac{d'_i-d_i}{3^{i+1}}=\sum_{i=0}^{q-1}\frac{d'_i-d_i}{3^{i+1}}+\sum_{i=0}^{10}\frac{x_i}{3^{i+1}}+\frac{1}{3^{q+11+1}}+\sum_{i=q+12}^{+\infty}\frac{d'_i-d_i}{3^{i+1}},$$
and it suffices to note that $\left|\sum_{i=q+12}^{+\infty}\frac{d'_i-d_i}{3^{i+1}}\right|\leq \frac{1}{3^{q+11+1}}$ and $\sum_{i=0}^{q-1}\frac{d'_i-d_i}{3^{i+1}}$ in the $3$-ary expansion of $d'-d$ affects only positions less than $q$  (similarly for $d''-d$).

Finally, by the observations from the beginning of this proof, we have
$$\left|(C-d)\cap (C-d')\cap (C-d'')\right|=\left|C\cap (C-(d'-d))\cap (C-(d''-d))\right|=0.$$ 
It follows that $C\in\h 2$.
\end{proof}

\begin{cor}
\label{x}
There is a compact Haar-$2$ set which is not a countable union of closed Haar-$1$ sets. Hence, $\overline{\h 2}\setminus\sigma\overline{\h 1}\neq\emptyset$.
\end{cor}

\begin{proof}
Consider the standard ternary Cantor set $C$. We already know that it is Haar-$2$. Observe that whenever $C\cap I\neq\emptyset$ for an open interval $I\subseteq\R$, we have $C\cap I\notin\h 1$ (this can be done using the same methods as in the previous proof, since $C\cap I$ contains an affine copy of $C$). It follows from Remark \ref{sigma} that $C\notin\sigma\overline{\h 1}$.
\end{proof}

\section{Haar-$n$ sets}

Now we move to Haar-$n$ sets for arbitrary $n\in\omega$. 

\begin{thm}
For each $n>0$ there are compact Haar-$1$ sets $A_0,\ldots,A_{n}$ such that $A_0\cup \ldots\cup A_{n}$ is Haar-$(n+1)$ but not Haar-$n$. In particular, $\overline{\h (n+1)}\setminus\overline{\h n}\neq\emptyset$.
\end{thm}

\begin{proof}
Fix $n>0$ and define: 
$$A_m=\left\{\sum_{k\in\omega}\frac{i_k}{5^{k+1}}:\ \forall_{k\in\omega}\ i_{k}\in 5\ \wedge\ i_{(n+1)k+m}\in\{0,4\}\right\}$$
for $m=0,\ldots,n$.

Observe that each $A_m$ is closed and in $\h 1$, by Proposition \ref{null-1}, since $d_{m,k}=\frac{5}{2\cdot 5^{(n+1)k+m+1}}$ is a converging to $0$ sequence with $(A_m-d_{m,k})\cap A_m=\emptyset$ for all $k$.

By Theorem \ref{sumyHn}, $A_0\cup \ldots\cup A_{n}$ is Haar-$(n+1)$. We will show that $A_0\cup \ldots\cup A_{n}$ is not Haar-$n$. 

Fix any Cantor set $D\subseteq\mathbb{R}$ and find $x_0<\ldots<x_{n}\in D$ such that $x_{n}-x_0\in (0,\frac{1}{5})$. Denote $d_m=x_m-x_0$ for $m=1,\ldots,n$ and let $(d_{m,i})_i\subseteq 5^\omega$ be the $5$-ary expansion of $d_m$. Note that $d_{m,0}=0$ for all $m$ (as $d_m\leq d_n<\frac{1}{5}$). We will construct $5$-ary expansions of $a_m\in A_m$, for $m=0,\ldots,n$, such that $a_0=a_1-d_1=\ldots=a_{n}-d_{n}$. It will follow that: 
$$a_0-x_0\in\bigcap_{m=0}^{n} (A_m-x_m)\subseteq \bigcap_{m=0}^{n} ((A_0\cup \ldots\cup A_{n})-x_m).$$

Define:
$$a_{0,i}=\left\{\begin{array}{ll}
0 & \textrm{if }(n+1)|i,\\
4-d_{m,i} & \textrm{if }i\equiv m \mod (n+1),
\end{array}\right.$$
and $\bar{a}_{m,i}=a_{0,i}+d_{m,i}\in\{0,\ldots,8\}$ for all $m=1,\ldots,n$ and $i\in\omega$. Note that $\bar{a}_{0,0}=0$ and $a_{m,0}=0$ for all $m$. We need to improve the sequences $(\bar{a}_{m,i})_i$ so that they will have only values $\leq 4$. 

Fix $m=1,\ldots,n$ and denote $B_m=\{i\in\omega:\ \bar{a}_{m,i}>4\}$. Observe that $i(n+1)\notin B_m$ (as $\bar{a}_{m,i(n+1)}=d_{m,i(n+1)}$) and $(n+1)i+m\notin B_m$ (as $\bar{a}_{m,(n+1)i+m}=4$), for all $i$. 

Suppose first that $B_m$ is infinite. Inductively define a sequence $(\beta_{m,i})_i\in 2^\omega$ by:
\begin{itemize}
\item $\beta_{m,i}=1$ for all $i\in B_m-1$;
\item $\beta_{m,i}=1$ iff $\beta_{m,i+1}=1$ and $\bar{a}_{m,i+1}+1>4$, for all other $i$;
\end{itemize}
(this can be easily done by considering all $i\in B_m-1$ in the first inductive step, all $i\in (B_m-2)\setminus (B_m-1)$ in the second inductive step, all $i\in (B_m-3)\setminus ((B_m-1)\cup(B_m-2))$ in the third inductive step, etc.). The sequence $(\beta_{m,i})_i$ corresponds to carrying in the standard algorithm for addition -- $\beta_{m,i}$ will indicate whether $\bar{a}_{m,i}$ should be enlarged by one (this case corresponds to $\beta_{m,i}=1$) or not ($\beta_{m,i}=0$).

We are ready to define the improved sequence:
$$a_{m,i}=(\bar{a}_{m,i}+\beta_{m,i})\mod 5.$$
Recall that $\bar{a}_{m,(n+1)i+m}=4$ for all $i$. Thus, it is easy to check that $a_m=\sum_{i\in\omega}\frac{a_{m,i}}{5^{i+1}}\in A_m$ as $a_{m,(n+1)i+m}$ is either equal to $4$ (if $\beta_{m,(n+1)i+m}=0$) or to $0$ (if $\beta_{m,(n+1)i+m}=1$). 

To finish the case of $B_m$ being infinite, we will show that $\sum_{i\in\omega}\frac{a_{m,i}}{5^{i+1}}=\sum_{i\in\omega}\frac{\bar{a}_{m,i}}{5^{i+1}}$. As $\bar{a}_{m,i}=a_{0,i}+d_{m,i}$ for all $i$, it will follow that $a_0=\sum_{i\in\omega}\frac{a_{0,i}}{5^{i+1}}$ is equal to $a_m-d_m$. 

Observe that $a_{m,b}=(\bar{a}_{m,b} \mod 5)+\beta_{m,b}$ for all $b\in B_m$, since $5\leq\bar{a}_{m,b}\leq 8$. Suppose that $b$ and $b'$ are two consecutive elements of $B_m$. We will show that
\begin{equation}
\label{eq}
\sum_{i\in(b,b']}\frac{\bar{a}_{m,i}}{5^{i+1}}=\frac{\beta_{m,b}}{5^{b+1}}+\sum_{i\in(b,b')}\frac{a_{m,i}}{5^{i+1}}+\frac{(\bar{a}_{m,b'}\mod 5)}{5^{b'+1}}.
\end{equation}
There are two possible cases:
\begin{itemize}
	\item If $\beta_{m,b}=1$, then it must be the case that $\beta_{m,i}=1$ and $\bar{a}_{m,i}=4$ (so $a_{m,i}=0$) for all $i\in(b,b')$. Observe that 
$$\sum_{i\in(b,b')}\frac{4}{5^{i+1}}=\frac{4}{5^{b+2}}\sum_{i=0}^{b'-b-2}\frac{1}{5^{i}}=\frac{4}{5^{b+2}}\frac{1-\frac{1}{5^{b'-b-1}}}{1-\frac{1}{5}}=\frac{1}{5^{b+1}}-\frac{1}{5^{b'}}.$$
Thus, the left hand side of equation (\ref{eq}) is equal to
$$\sum_{i\in(b,b')}\frac{4}{5^{i+1}}+\frac{\bar{a}_{m,b'}}{5^{b'+1}}=\sum_{i\in(b,b')}\frac{4}{5^{i+1}}+\frac{5}{5^{b'+1}}+\frac{(\bar{a}_{m,b'}\mod 5)}{5^{b'+1}}=$$
$$=\frac{1}{5^{b+1}}-\frac{1}{5^{b'}}+\frac{1}{5^{b'}}+\frac{(\bar{a}_{m,b'}\mod 5)}{5^{b'+1}}=\frac{1}{5^{b+1}}+\frac{(\bar{a}_{m,b'}\mod 5)}{5^{b'+1}},$$
which is the right hand side of equation (\ref{eq}).
	\item If $\beta_{m,b}=0$, then there is $j\in(b,b')$ such that $\beta_{m,i}=0$ for all $i\in[b,j)$ and $\beta_{m,i}=1$ for all $i\in[j,b')$. Then 
	\begin{itemize}
		\item $a_{m,i}=\bar{a}_{m,i}$ for all $i\in(b,j)$;
		\item $\bar{a}_{m,j}+\beta_{m,j}=\bar{a}_{m,j}+1\leq 4$ (so $a_{m,j}=\bar{a}_{m,j}+1$);
		\item $\bar{a}_{m,i}=4$ (so $a_{m,i}=0$) for all $i\in(j,b')$. 
	\end{itemize}
Note that $\sum_{i\in(j,b')}\frac{4}{5^{i+1}}=\frac{1}{5^{j+1}}-\frac{1}{5^{b'}}$. Now the left hand side of equation (\ref{eq}) is equal to
$$\sum_{i\in(b,j]}\frac{\bar{a}_{m,i}}{5^{i+1}}+\sum_{i\in(j,b')}\frac{4}{5^{i+1}}+\frac{\bar{a}_{m,b'}}{5^{b'+1}}=$$
$$=\sum_{i\in(b,j]}\frac{\bar{a}_{m,i}}{5^{i+1}}+\sum_{i\in(j,b')}\frac{4}{5^{i+1}}+\frac{5}{5^{b'+1}}+\frac{(\bar{a}_{m,b'}\mod 5)}{5^{b'+1}}=$$
$$=\sum_{i\in(b,j)}\frac{a_{m,i}}{5^{i+1}}+\left(\frac{a_{m,j}}{5^{j+1}}-\frac{1}{5^{j+1}}\right)+\left(\frac{1}{5^{j+1}}-\frac{1}{5^{b'}}\right)+\frac{1}{5^{b'}}+\frac{(\bar{a}_{m,b'}\mod 5)}{5^{b'+1}}=$$
$$=\sum_{i\in(b,j]}\frac{a_{m,i}}{5^{i+1}}+\frac{(\bar{a}_{m,b'}\mod 5)}{5^{b'+1}},$$
which is the right hand side of equation (\ref{eq}).
\end{itemize}

We have shown equation (\ref{eq}) for all pairs $(b,b')$ of two consecutive elements of $B_m$. Moreover, similarly as in the second case, one can show that
$$\sum_{i\in[0,\min B_m)}\frac{a_{m,i}}{5^{i+1}}+\frac{(\bar{a}_{m,\min B_m}\mod 5)}{5^{\min B_m+1}}=\sum_{i\in[0,\min B_m]}\frac{\bar{a}_{m,i}}{5^{i+1}}$$
(if $\beta_{m,i}=1$ for all $i\in[0,\min B_m)$, then the role of $j$ from the second case is played by $0$, as $\bar{a}_{m,0}+\beta_{m,0}=0+\beta_{m,0}\leq 1$). Therefore, $a_m=\sum_{i\in\omega}\frac{\bar{a}_{m,i}}{5^{i+1}}$ and the case of $B_m$ being infinite is finished.

In the case $B_m$ is finite, it suffices to put $\beta_{m,i}=0$ for all $i\geq\max B_m$ and proceed analogously. 

By conducting in the same way for all $m=1,\ldots,n$, we get the desired points $a_1,\ldots,a_n$ and we are done.
\end{proof}

The next corollary is an immediate consequence of the above theorem.

\begin{cor}
For each $n\in\omega\setminus\{0\}$ the family $\h n$ does not form an ideal in $\mathbb{R}$.
\end{cor}

In the next section we will need two lemmas. 

\begin{df}
\label{Cl}
Let $q_0=25$ and $q_n=q_{n-1}\cdot 25\cdot 3^n$, for $n\geq 1$. Denote by $S$ the set of all $(x_i)_i\in\omega^\omega$ satisfying $x_i\in 25\cdot 3^i$, for all $i\in\omega$. For any $l\in\omega$ we put $m_l=\frac{25\cdot 3^l-1}{2}$ and denote: 
$$C_l=\left\{\sum_{k\in\omega}\frac{i_k}{q_k}:\ (i_k)_k\in S\ \wedge\ \forall_{k\geq l}\ \lfloor\frac{i_k}{3^{i-l}}\rfloor\neq m_l\right\}.$$
\end{df}

\begin{lem}
\label{lem1}
For every $l\in\omega\setminus \{0\}$ and every open interval $I$ such that $C_l\cap I\neq \emptyset$, the set $C_l\cap I$ is not Haar-$(m_l-1)$.
\end{lem}

\begin{proof}
Since $C_l\cap I\neq \emptyset$, there is $s\in\omega^{<\omega}$ such that $s_k\in 25\cdot 3^k$, for all $k<\text{lh}(s)$, $\sum_{k<\text{lh}(s)}\frac{s_k}{q_k}\in C_l\cap I$ and $\sum_{k<\text{lh}(s)}\frac{s_k}{q_k}+\frac{1}{q_{\text{lh}(s)}}\in C_l\cap I$. 

Fix any Cantor set $D\subseteq\mathbb{R}$ and pick $x_0<\ldots<x_{m_l-1}\in D$ such that $x_{m_l-1}-x_0<\frac{1}{q_{\text{lh}(s)-1}}$. Denote $d_i=x_i-x_0$ for $i=1,\ldots,m_l-1$. We will show that there is $c\in (C_l\cap I)\cap \bigcap_{i=1}^{m_l-1}((C_l\cap I)-d_i)$. It will follow that $c-x_0\in\bigcap_{i=0}^{m_l-1}((C_l\cap I)-x_i)$.

Let $(d_{i,j})_j\in S$ be such that $d_i=\sum_{j\in\omega}\frac{d_{i,j}}{q_j}$, for each $i=1,\ldots,m_l-1$ (note that $d_{i,j}=0$ for all $i$ and $j<\text{lh}(s)$ by $x_0<\ldots<x_{m_l-1}$ and $x_{m_l-1}-x_0<\frac{1}{q_{\text{lh}(s)-1}}$). Without loss of generality we assume that
$d_{i,j}\neq 25\cdot 3^j-1$ for infinitely many $j$. We will inductively construct $(c_i)_i\in S$. We start with $c_i=s_i$ for all $i<\text{lh}(s)$. Suppose now that $n\geq\text{lh}(s)$ and $c_i$, for all $i<n$, are already defined. Denote:
$$Z^n_l\left\{\sum_{k\in\omega}\frac{i_k}{q_k}:\ (i_k)_k\in S\ \wedge\ \exists_k^\infty\ i_k\neq 25\cdot 3^k-1\ \wedge\ \lfloor\frac{i_k}{3^{i-l}}\rfloor=m_l\right\}$$
(the symbol $\exists^\infty_k$ stands for "there are infinitely many $k$"). Pick such $c_n\in 25\cdot 3^n$ that $\lfloor\frac{c_n}{3^{n-l}}\rfloor\neq m_l$ and neither $\sum_{j=0}^n \frac{c_j}{q_j}$ nor $\sum_{j=0}^n \frac{c_j}{q_j}+\frac{1}{q_n}$ belongs to the set
$$\bigcup_{i=1}^{m_l-1}\left(Z^n_l-\sum_{j=0}^n \frac{d_{i,j}}{q_j}\right).$$
This is possible by the Pigeonhole Principle, since 
$$\left|\left\{j\in 25\cdot 3^n:\ \lfloor\frac{j}{3^{n-l}}\rfloor\neq m_l\right\}\right|=25\cdot 3^n - 3^{n-l},$$
each set of the form $Z^n_l-\sum_{j=0}^n \frac{d_{i,j}}{q_j}$ excludes at most $3^{n-l}+1$ possible values of $c_n$ and
$$(3^{n-l}+1)\cdot (m_l-1)=(3^{n-l}+1)\cdot\left(\frac{25\cdot 3^l-1}{2} -1\right)=$$ $$=\frac{25\cdot 3^n}{2}-\frac{3}{2}\cdot 3^{n-l}+\frac{25\cdot 3^l-1}{2}-1<\frac{25\cdot 3^n}{2}-\frac{3}{2}\cdot 3^{n-l}+\frac{25\cdot 3^n}{2}-1<25\cdot 3^n-3^{n-l}.$$

Once the induction is completed, observe that $c=\sum_{j\in\omega} \frac{c_j}{q_j}$ is an element of $C_l\cap I$. We will show that it is also an element of $(C_l\cap I)-d_i$ for each $i=1,\ldots,m_l-1$. This will end the proof. 

Fix $i\in\{1,\ldots,m_l-1\}$ and suppose to the contrary that $c\notin (C_l\cap I)-d_i$. Then $c\in\bigcup_{n\geq l}((Z^n_l\cap I)-d_i)$, since $c+d_i\in[0,1]\cap I$ (as $c_i=s_i$, for all $i<\text{lh}(s)$, $d_{i,j}=0$, for all $j<\text{lh}(s)$, and $\sum_{k<\text{lh}(s)}\frac{s_k}{q_k}+\frac{1}{q_{\text{lh}(s)}}\in C_l\cap I\subseteq[0,1]\cap I$). Thus, there is $n_0\geq l$ such that $c\in Z_{n_0}-d_i$. Observe that: 
$$c\in \left[\sum_{j=0}^{n_0} \frac{c_j}{q_j},\sum_{j=0}^{n_0} \frac{c_j}{q_j}+\frac{1}{q_{n_0}}\right]$$ 
and denote the above interval by $J$. By the choice of $c_{n_0}$, we know that endpoints of $J$ do not belong to the set $Y=Z^{n_0}_l-\sum_{j=0}^{n_0} \frac{d_{i,j}}{q_j}$, which is a union of intervals of the form $\left[\frac{p}{q_{n_0}},\frac{p+1}{q_{n_0}}\right)$, for some integer $p$. Thus, $J$ and $Y$ are disjoint. Moreover, the distance between $\max J$ and $Y$ is at least $\frac{1}{q_{n_0}}$. Now it suffices to observe that $\sum_{j>n_0} \frac{d_{i,j}}{q_j}\in\left[0,\frac{1}{q_{n_0}}\right)$ (recall that $d_{i,j}\neq 25\cdot 3^j-1$ for infinitely many $j$). Hence, $J$ is disjoint with $ Z^{n_0}_l-d_i$, a contradiction.
\end{proof}

\begin{lem}
\label{lem2}
For every $l\in\omega$ the set $C_l$ is Haar-$(2m_l+1)$.
\end{lem}

\begin{proof}
This proof is similar to the proof of Theorem \ref{one-third}.

At first fix any $k>l$ and define 
$$x_j=\frac{j\cdot 3^{k-l}}{q_k}+\frac{25\cdot 3^{k+1}-1-j}{q_{k+1}},$$ 
for $j=0,\ldots,2m_l$, and $x_{2m_l+1}=0$. Note that $25\cdot 3^k-1-j\in 25\cdot 3^k$ for each $j$. Denote
$$W_l^k=\left\{\sum_{j\in\omega}\frac{i_j}{q_j}:\ \forall_{j}\ i_j\in 25\cdot 3^j\ \wedge\ \lfloor\frac{i_k}{3^{k-l}}\rfloor\neq m_l\right\}$$
and observe that $\bigcap_{j\in 2m_l+2}(W^k_l-x_j)=\emptyset$ (basically, each gap in the set $W^k_l=W^k_l-x_{2m_l+1}$ is an open interval of length $\frac{3^{k-l}}{q_k}$ and the distance between two consecutive such gaps is $\frac{25\cdot 3^k-3^{k-l}}{q_k}$, so using $\frac{25\cdot 3^k-3^{k-l}}{3^{k-l}}+1=25\cdot 3^l-1+1=2m_l+1$ translations of the right hand side gap by $x_j$, for $j\in 2m_l+1$, we are able to cover the whole closed interval beginning in the end of the left hand side gap and ending in the beginning of the right hand side gap). Thus, $\bigcap_{j\in 2m_l+2}(C_l-x_j)$ is empty as well, since $C_l\subseteq W^k_l$.

We can further generalize this observation: for any $(a_{0,j})_j,\ldots,(a_{2m_l+1,j})_j\in \omega^k$ such that $a_{i,j}\in 25\cdot 3^j$ for all $i$ and $j$, if $a_i=\sum_{j\in k}\frac{a_{i,j}}{q_j}$, for each $i=0,\ldots,2m+1$, then we have $\bigcap_{j\in 2m_l+2}(C_l-(a_j+x_j))=\emptyset$ (since $\bigcap_{j\in 2m_l+2}(C_l-a_j)\subseteq\bigcap_{j\in 2m_l+2}(W^k_l-a_j)$ and the latter intersection is equal to $W^k_l-a_0$ intersected with some interval). Moreover, $\bigcap_{j\in 2m_l+2}(C_l-(a_j+x_j+a'_j))=\emptyset$ for any $a'_0,\ldots,a'_{2m_l+1}\in[0,\frac{1}{q_{k+2}})$ (since for each $i\in 2m_l+2$ the distance between $C_l-(a_i+x_i)$ and $\bigcap_{j\neq i}(C_l-(a_j+x_j))$ is at least $\frac{1}{q_{k+2}}$). 

Now, using the above, we construct a Cantor set $D$ witnessing that $C_l$ is Haar-$(2m_l+1)$. Let $n_l\in\omega$ be minimal such that $2^{n_l}\geq 2m_l+2$. Define $k_{n_l-1}=0$ and $k_{n}=k_{n-1}+2\left(\genfrac{}{}{0pt}{}{2^n}{2m_l+2}\right)$, for $n\geq n_l$ (where $\left(\genfrac{}{}{0pt}{}{2^n}{2m_l+2}\right)$ denotes the number of $(2m_l+2)$-combinations of the set $2^n$). Pick inductively finite sequences $\bar{d}^s=(d^{s}_i)_i\in \omega^{k_n-k_{n-1}}$ with $d^{s}_i\in 25\cdot 3^{k_{n-1}+i}$, for every $s\in 2^n$ and $n\geq n_l$. At the end $D$ will consist of all points of the form $\sum_{n\geq n_l}d_{x|n}$ for some $x\in 2^\omega$, where $d_{s}=\sum_{i=0}^{k_n-k_{n-1}-1}\frac{d^s_{i}}{q_{k_{n-1}+i}}$, for $s\in 2^n$, $n\geq n_l$. In the $n$th step define $\bar{d}^s$, for each $s=(s_i)\in 2^n$, in such a way that:
\begin{itemize}
	\item $\bar{d}^s$ is a concatenation of $\left(\genfrac{}{}{0pt}{}{2^n}{2m_l+2}\right)$ sequences of length $2$;	
	\item for every pairwise distinct $s_0,\ldots,s_{2m_l+1}\in 2^n$ (note that there are exactly $\left(\genfrac{}{}{0pt}{}{2^n}{2m_l+2}\right)$ ways to choose $s_0,\ldots,s_{2m_l+1}$) we can find $F=\{2p,2p+1\}$, for some $p<\left(\genfrac{}{}{0pt}{}{2^n}{2m_l+2}\right)$, such that: 
$$\{(\bar{d}^{s_j}_i)_{i\in F}:\ j\in 2m_l+2\}=$$
$$=\{(0,0)\}\cup\{(j\cdot 3^{k_{n-1}+2p-l},25\cdot 3^{k_{n-1}+2p+1}-1-j):\ j\in 2m_l+1\}.$$
\end{itemize}
By the construction, given any pairwise distinct $d_0,\ldots,d_{2m_l+1}\in D$, we have $\bigcap_{j=0}^{2m_l+1}(C_l-d_j)=\emptyset$. Hence, $C_l$ is Haar-$(2m_l+1)$.
\end{proof}

\section{Haar-finite sets}
\label{Haar-finite}

Recall that no countable union of Haar-finite sets can be the whole space, as Haar-finite sets are Haar-null and Haar-meager. In this section we show that a union of two Haar-finite sets does not have to be Haar-finite. This answers a question posed by Swaczyna during his talk on XLI Summer Symposium in Real Analysis (Wooster, 2017). Note also that, as a consequence of the next theorem, $\I$ being an ideal does not guarantee that $\h\I$ is an ideal.

\begin{thm}
\label{not ideal}
The family of Haar-finite subsets of $\R$ is not an ideal.
\end{thm}

\begin{proof}
We need to construct $A,B\in\h\Fin$ and $X\notin\h\Fin$ with $A\cup B=X$. We will use notation from Definition \ref{Cl}. The proof consists of $5$ steps. 

\textbf{Step $1$: Construction of the set $X$. }

Define a function $\phi\colon S\to[0,1]$ by $\phi((x_i)_i)=\sum_{i\in\omega}\frac{x_i}{q_i}$.

Define $L_0=\{8,10\}$ and let $X_0$ be the set of all points $\phi((x_i)_i)$ such that $(x_i)_i\in S$, $x_0\notin L_0$ and $\phi((x_i)_i)\in C_0$. Equivalently, $X_0=C_{0}\setminus((\frac{8}{25},\frac{9}{25})\cup(\frac{10}{25},\frac{11}{25}))$. Thus, the set $L_0$ corresponds to gaps in the set $C_0$ -- we will fill those gaps with subsets of the sets $C_m$. Observe that $X_0$ is Haar-$(2m_0+1)$ by Lemma \ref{lem2}. 

For each $n\geq 1$ let $L_n=(3L_{n-1})\cup(3L_{n-1}+2)$. For each $n$ let $X_n$ consist of all points of the form $\phi((x_i)_i)$ such that $(x_i)_i\in S$, $\phi((x_i)_i)\in C_n$, $x_m\in L_m$, for all $m<n$, and $x_n\notin L_{n}$. Equivalently, 
$$X_n=C_n\cap\left(\bigcup_{l\in L_{n-1}}\left[\frac{l}{q_{n-1}},\frac{l+1}{q_{n-1}}\right]\right)\setminus\left(\bigcup_{l'\in L_{n}}\left(\frac{l'}{q_{n}},\frac{l'+1}{q_{n}}\right)\right).$$

Observe that $X_n$ is Haar-$(2m_n+1)$ by Lemma \ref{lem2}. Actually, even $\bigcup_{i\leq n}X_i$ is Haar-$(2m_n+1)$ (as $X_i\subseteq C_n$ for each $i\leq n$).

Notice that none $L_n$ contains two consecutive integers. One can inductively show that $\max L_n<m_n-1$, for all $n$ (in particular, the sets $X_n$ are well defined). Indeed, $\max L_0=10<11=m_0-1$ and we have: 
$$\max L_{n+1}=3\cdot\max L_{n}+2<3\cdot\left(\frac{25\cdot 3^n-1}{2}-1\right)+2<\frac{25\cdot 3^{n+1}-1}{2}-1=m_{n+1}-1.$$
Moreover, $\min L_n>\frac{m_n+1}{2}$, for all $n$, as $\min L_0=8>\frac{13}{2}=\frac{m_0+1}{2}$ and:
$$\min L_{n+1}=3\cdot\min L_{n}>3\cdot\frac{\frac{25\cdot 3^n-1}{2}+1}{2}>\frac{\frac{25\cdot 3^{n+1}-1}{2}+1}{2}=\frac{m_{n+1}+1}{2}.$$
In particular, $\frac{m_n}{2}<\min L_n<\max L_n<m_n$ for all $n$. 

The set $X$ is given by: 
$$X=\left(\bigcup_{n\in\omega}X_n\right)\cup\left\{\phi((x_i)_i):\ \forall_{i\in\omega}\ x_i\in L_i\right\}.$$

The purpose of adding the set $\left\{\phi((x_i)_i):\ \forall_{i\in\omega}\ x_i\in L_i\right\}$ is connected with showing that $X$ is not Haar-finite -- in the next step we will find an element of this set belonging to infinitely many translations of $X$. 

Observe that $X$ is compact. 

\textbf{Step $2$: $X$ is not Haar-finite. }

Before proving that $X$ is not Haar-finite, let us construct some auxiliary sets: for each $n,k\in\omega$ with $k\geq n$ let $T^k_n$ consist of all points of the form $\phi((x_i)_i)$ such that $(x_i)_i\in S$, $\phi((x_i)_i)\in Z^k_n$ (where $Z^k_n$ is as in the proof of Lemma \ref{lem1}) and $x_m\in L_m$, for all $m<n$. Define $T_n=\bigcup_{k\geq n}T^k_n$ for all $n$. Notice that $[0,1]\setminus T_n\subseteq X$ for each $n$.

Now we proceed to proving that $X$ is not Haar-finite. Fix any Cantor set $C\subseteq\mathbb{R}$ and choose a decreasing sequence $(c'_i)_i\subseteq C$ such that $c_0=c'_0-\inf C<\frac{12}{25}=\frac{m_0}{q_0}$ and $c_i=c'_i-\inf C<\frac{1}{q_{i-1}}$ for all $i\in\omega$, $i\geq 1$. 

Let $(c_{i,j})_j\in S$ be such that $c_i=\phi((c_{i,j})_j)$, for each $i\in\omega$. We will inductively construct a sequence $(r_i)_i$ with $r_i\in L_i$, for each $i\in\omega$. We start by picking $r_0\in L_0$ such that neither $\frac{r_0}{q_0}$ nor $\frac{r_0+1}{q_0}$ belongs to the set $T^0_0-\frac{c_{0,0}}{q_0}$. This is possible by the Pigeonhole Principle as $T^0_0-\frac{c_{0,0}}{q_0}$ (which is an interval of the form $\left[\frac{p}{q_0}.\frac{p+1}{q_0}\right)$ for some integer $p$) can exclude at most two consecutive values of $r_0\in q_0$, i.e., at most one element of $L_0=\{8,10\}$. Next, pick $r_1\in L_1$ such that neither $\sum_{j=0}^1 \frac{r_j}{q_j}$ nor $\sum_{j=0}^1 \frac{r_j}{q_j}+\frac{1}{q_1}$ belongs to the set 
$$\left(T^1_0-\sum_{j=0}^1 \frac{c_{0,j}}{q_j}\right)\cup\left(T^1_1-\sum_{j=0}^1 \frac{c_{1,j}}{q_j}\right).$$ 
Again, this is possible by the Pigeonhole Principle as $L_1=\{24,26,30,32\}$ and $T^1_0-\sum_{j=0}^1 \frac{c_{0,j}}{q_j}$ can exclude at most four consecutive values of $r_1\in 25\cdot 3$ (i.e., it cannot exclude simultaneously something from $\{24,26\}$ and something from $\{30,32\}$) while $T^1_1-\sum_{j=0}^1 \frac{c_{1,j}}{q_j}$ can exclude at most two consecutive values of $r_1\in 25\cdot 3$. Suppose now that $r_i$, for all $i<k$, are already defined. There is some $r_k\in L_k$ such that neither $\sum_{j=0}^k \frac{r_j}{q_j}$ nor $\sum_{j=0}^k \frac{r_j}{q_j}+\frac{1}{q_k}$ belongs to the set 
$$\bigcup_{i=0}^{k}\left(T^k_i-\sum_{j=0}^k \frac{c_{i,j}}{q_j}\right).$$
Similarly as above, each $T^k_i-\sum_{j=0}^k \frac{c_{i,j}}{q_j}$ excludes at most $3^{k-i}+1$ consecutive values of $r_k$ and one can inductively show that $T^k_0-\sum_{j=0}^k \frac{c_{0,j}}{q_j}$ excludes at most half of the set $L_k$, $T^k_1-\sum_{j=0}^k \frac{c_{1,j}}{q_j}$ excludes at most one-fourth of the set $L_k$ etc. (using the following observation: if there are $l,l'\in L_{k-1}$ with $(l,l')\cap L_{k-1}=\emptyset$ and $l'-l=3^j+1$, for some $j$, then $3l+2,3l'\in L_k$, $(3l+2,3l')\cap L_{k}=\emptyset$ and $(3l')-(3l+2)=3(3^{j}+1)-2=3^{j+1}+1$).

Once the induction is completed, observe that $r=\phi((r_i)_i)$ is an element of $X$. What is more, we will show that it belongs to $\bigcap_{i\in\omega}(X-c_i)$.

Fix $i\in\omega$. We will show that $r\notin T_i-c_i$. As $r+c_i\leq r+c_0<\frac{11}{25}+\frac{12}{25}<1$ (recall that $r_0\in L_0=\{8,10\}$) and $[0,1]\setminus T_i\subseteq X$, this will finish this step. 

Suppose to the contrary that $r\in T_i-c_i$. Thus, there is $k_0\geq i$ such that $r\in T_i^{k_0}-c_i$. Observe that:
$$r\in\left[\sum_{j=0}^{k_0} \frac{r_j}{q_j},\sum_{j=0}^{k_0} \frac{r_j}{q_j}+\frac{1}{q_{k_0}}\right].$$ 
Denote the above interval by $I$. By the choice of $r_{k_0}$, we know that endpoints of $I$ do not belong to the set $T^{k_0}_{i}-\sum_{j=0}^{k_0} \frac{c_{i,j}}{q_j}$, which is a union of intervals of the form $\left[\frac{p}{q_{k_0}},\frac{p+1}{q_{k_0}}\right)$, for some integer $p$. Thus, the distance between $\max I$ and $T^{k_0}_{i}-\sum_{j=0}^{k_0} \frac{c_{i,j}}{q_j}$ is at least $\frac{1}{q_{k_0}}$. Now it suffices to observe that $\sum_{j>k_0} \frac{c_{i,j}}{q_j}<\frac{1}{q_{k_0}}$. Hence, $I$ is disjoint with $T^{k_0}_{i}-\sum_{j=0}^{k_0} \frac{c_{i,j}}{q_j}$, a contradiction.

\textbf{Step $3$: Partition of $X$ into two sets $A$ and $B$.}

Let $B_0$ consist of those $\phi((x_i)_i)$ belonging to $X_0$ such that 
$$\text{dist}\left(\{x_0-\lfloor\frac{m_0}{2}\rfloor\},L_0\right)\leq 1.$$ 
Equivalently, $B_0=X_0\cap \left[\frac{7}{13},\frac{11}{13}\right]$.

For each $n\in\omega\setminus\{0\}$ let $B_n$ consist of those $\phi((x_i)_i)$ belonging to $X_n$ such that:
$$\text{dist}\left(\{x_n-\lfloor\frac{m_n}{2}\rfloor\},L_n\right)\leq 1$$ 
as well as those $\phi((x_i)_i)$ belonging to $X_{n}$ such that:
$$\text{dist}\left(\{x_n+\lfloor\frac{m_n}{2}\rfloor\},L_n\right)\leq 1.$$ 

Define $B=\bigcup_{n\in\omega}B_n$ and $A=\overline{X\setminus B}$. Observe that $B$ is $\mathtt{F_\sigma}$ while $A$ is compact. Moreover, $A\setminus(X\setminus B)$ consists of countably many points -- endpoints of intervals used in the definitions of sets $B_n$.

The idea is to make gaps in the set $A$ so that each set $\left\{\phi((x_i)_i):\ \forall_{i>n}\ x_i\in L_i\right\}$, for $n\in\omega$, can be translated (by $\frac{\lfloor\frac{m_n}{2}\rfloor}{q_n}$) into one of those gaps. Then the intersection of $A$ with the mentioned translation $A-\frac{\lfloor\frac{m_n}{2}\rfloor}{q_n}$ will be in fact a subset of $\bigcup_{i\leq n}C_i$ intersected with $\left(\bigcup_{i\leq n}C_i\right)-\frac{\lfloor\frac{m_n}{2}\rfloor}{q_n}$. As we already know that $\bigcup_{i\leq n}C_i$ is Haar-$(2m_n+1)$, the above observation will help us show that $A$ is Haar-finite. In the next step we will state this fact more precisely.

\textbf{Step $4$: $A$ is Haar-finite. }

We will construct a Cantor set $D$ witnessing that $A$ is Haar-finite. 

Define $k_0=0$ and $k_{n+1}=k_n+1+2\left(\genfrac{}{}{0pt}{}{2^n}{2w_n+2}\right)$, for $n>1$, where $(w_n)_n$ is any sequence onto $\{m_l:\ l\in\omega\}$ with $2w_n+2<2^n$, for all $n$, and $\{n\in\omega:\ w_n=m_l\}$ infinite for each $l\in\omega$. Pick inductively finite sequences $\bar{d}^s=(d^{s}_i)_i\in \omega^{k_n-k_{n-1}}$, for every $s\in 2^n$ and $n\geq 1$. At the end $D$ will consist of all points of the form $\sum_{n\geq 1}d_{x|n}$ for some $x\in 2^\omega$, where
$$d_{s}=\sum_{i=0}^{k_n-k_{n-1}-1}\frac{d^s_{i}}{q_{k_{n-1}+i}}$$ 
for $s\in 2^n$. In the $n$th step define $\bar{d}^s$, for each $s=(s_i)\in 2^n$, in such a way that:
\begin{itemize}
	\item $\bar{d}^s$ is a concatenation of $(s_{n-1}\cdot \lfloor\frac{m_{k_{n-1}}}{2}\rfloor)$ and $\left(\genfrac{}{}{0pt}{}{2^n}{2w_n+2}\right)$ sequences of length $2$;	
	\item for every pairwise distinct $s_0,\ldots,s_{2w_n+1}\in 2^n$ we can find $F=\{1+2p,1+2p+1\}$, for some $p<\left(\genfrac{}{}{0pt}{}{2^n}{2w_n+2}\right)$, such that: 
$$\left\{(\bar{d}^{s_j}_i)_{i\in F}:\ j\in 2w_n+2\right\}=$$
$$=\{(0,0)\}\cup\{(j\cdot 3^{k_{n-1}+1+2p-l},25\cdot 3^{k_{n-1}+2p+2}-1-j):\ j\in 2w_n+1\},$$
where $l$ is such that $w_n=m_l$.
\end{itemize}

Now we will show that $D$ is as needed. Given any infinite $\{x_j:\ j\in\omega\}\subseteq 2^\omega$, denote $d_j=\sum_{n\geq 1}d_{x_j|n}\in D$, for each $j\in\omega$. Our goal is to prove that $\bigcap_{j\in\omega}(A-d_j)=\emptyset$. 

Let $\tilde{n}$ be minimal such that $x_{j_0}|\tilde{n}\neq x_{j_1}|\tilde{n}$ for some $j_0,j_1\in\omega$. Without loss of generality we can assume that $x_{j}|\tilde{n}=x_{j_1}|\tilde{n}$ for infinitely many $j\in\omega$. Denote $\tilde{k}=k_{\tilde{n}-1}$. The crucial observation is the following:
$$(A-d_{j_0})\cap(A-d_{j})=\left(\left(\bigcup_{l\leq \tilde{k}}X_l\right)-d_{j_0}\right)\cap\left(\left(\bigcup_{l\leq \tilde{k}}X_l\right)-d_{j}\right),$$
whenever $x_{j}|\tilde{n}=x_{j_1}|\tilde{n}$. Indeed, 
$$\left(A\cap \bigcup_{l>\tilde{k}}X_l\right)-\frac{\lfloor\frac{m_{\tilde{k}}}{2}\rfloor}{q_{\tilde{k}}}\subseteq\left\{\phi((x_i)_i):\ (x_i)_i\in S\ \wedge\ \forall_{l\leq \tilde{k}}x_{l}\in L_{l}\right\}-\frac{\lfloor\frac{m_{\tilde{k}}}{2}\rfloor}{q_{\tilde{k}}}=$$
$$=\left\{\phi((x_i)_i):\ (x_i)_i\in S\ \wedge\ \forall_{l<\tilde{k}}x_{l}\in L_{l}\ \wedge\ x_{\tilde{k}}\in L_{\tilde{k}}-\lfloor\frac{m_{\tilde{k}}}{2}\rfloor\right\}.$$
Denote the latter set by $Y$ and observe that $Y\cap\bigcup_{l>\tilde{k}}X_l=\emptyset$ (as $L_{\tilde{k}}\cap(L_{\tilde{k}}-\lfloor\frac{m_{\tilde{k}}}{2}\rfloor)=\emptyset$ by $\max L_{\tilde{k}}<m_{\tilde{k}}-1$ and $\min L_{\tilde{k}}>\frac{m_{\tilde{k}}+1}{2}$) and $Y\cap\bigcup_{l<\tilde{k}}X_l=\emptyset$ (as $\min L_{\tilde{k}}-\lfloor\frac{m_{\tilde{k}}}{2}\rfloor>0$). Thus, $Y$ can intersect only $A\cap X_{\tilde{k}}$. However, all $\phi((x_i)_i)$ belonging to $X_{\tilde{k}}$ such that $x_{\tilde{k}}\in L_{\tilde{k}}-\lfloor\frac{m_{\tilde{k}}}{2}\rfloor$ are in $B_{\tilde{k}}$, hence not in $A$. Therefore, $\left(A\cap \bigcup_{l>\tilde{k}}X_l\right)-\frac{\lfloor\frac{m_{\tilde{k}}}{2}\rfloor}{q_{\tilde{k}}}$ does not intersect $A$. Actually, the above reasoning shows that the distance between those two sets is at least $\frac{1}{q_{\tilde{k}}}$. Hence, $\left(A\cap \bigcup_{l>\tilde{k}}X_l\right)-d_j$ is disjoint with $A-d_{j'}$ whenever $x_j|\tilde{n}$ ends with $1$ and $x_{j'}|\tilde{n}$ ends with $0$. A similar reasoning shows that $\left(A\cap \bigcup_{l>\tilde{k}}X_l\right)-d_{j'}$ is disjoint with $A-d_{j}$ whenever $x_j|\tilde{n}$ ends with $1$ and $x_{j'}|\tilde{n}$ ends with $0$.

Recall that $\bigcup_{i\leq\tilde{k}}X_i\subseteq C_{\tilde{k}}$ and $C_{\tilde{k}}$ is Haar-$(2m_{\tilde{k}}+1)$. Let $j_2,\ldots,j_{2m_{\tilde{k}}+2}$ be pairwise distinct and such that $x_{j_i}|\tilde{n}=x_{j_1}|\tilde{n}$ for $i=1,\ldots,2m_{\tilde{k}}+2$. Then, if $\tilde{m}\geq\tilde{n}$ is such that $x_{j_i}|\tilde{m}$, for $i=1,\ldots,2m_{\tilde{k}}+2$, are pairwise distinct and $w_{\tilde{m}}=m_{\tilde{k}}$, then for some $F=\{1+2p,1+2p+1\}$, where $p<\left(\genfrac{}{}{0pt}{}{2^n}{2w_{\tilde{m}}+2}\right)$, we have:
$$\left\{(\bar{d}^{x_{j_i}|\tilde{m}}_i)_{i\in F}:\ j=1,\ldots,2m_{\tilde{k}}+2\right\}=$$
$$=\{(0,0)\}\cup\{(j\cdot 3^{k_{\tilde{m}-1}+1+2p-\tilde{k}},25\cdot 3^{k_{\tilde{m}-1}+2p+2}-1-j):\ j\in 2w_{\tilde{m}}+1\}.$$
Thus, similarly as in the proof of Lemma \ref{lem1}, one can show that 
$$\bigcap_{i=1}^{2m_{\tilde{k}}+2}\left(\left(\bigcup_{l\leq \tilde{k}}X_l\right)-d_{j_i}\right)\subseteq \bigcap_{i=1}^{2m_{\tilde{k}}+2} (C_{\tilde{k}}-d_{j_i})=\emptyset.$$ 

Therefore,
$$\bigcap_{i\leq 2m_{\tilde{k}}+2}\left(A-d_{j_i}\right)=\bigcap_{i=1}^{2m_{\tilde{k}}+2}\left(\left(A-d_{j_i}\right)\cap\left(A-d_{j_0}\right)\right)=$$
$$=\bigcap_{i=1}^{2m_{\tilde{k}}+2}\left(\left(\left(\bigcup_{l\leq\tilde{k}}X_l\right)-d_{j_i}\right)\cap\left(\left(\bigcup_{l\leq\tilde{k}}X_l\right)-d_{j_0}\right)\right)\subseteq$$
$$\subseteq\bigcap_{i=1}^{2m_{\tilde{k}}+2}\left(\left(\bigcup_{l\leq \tilde{k}}X_l\right)-d_{j_i}\right)=\emptyset.$$

\textbf{Step $5$: $B$ is Haar-finite. }

We will construct a Cantor set $E$ witnessing that $B$ is Haar-finite. 

Define $u_0=0$ and $u_{n+1}=u_n+2+2\left(\genfrac{}{}{0pt}{}{2^n}{2w_n+2}\right)$, for $n>1$, where $(w_n)_n$ is as in the previous step. Inductively pick finite sequences $\bar{e}^s=(e^{s}_i)\in \omega^{u_n-u_{n-1}}$, for every $s\in 2^n$ and $n\geq 1$: $\bar{e}^s$, for each $s=(s_i)\in 2^n$, is such that:
\begin{itemize}
	\item $\bar{e}^s$ is a concatenation of $(s_{n-1},0)$ and $\left(\genfrac{}{}{0pt}{}{2^n}{2w_n+2}\right)$ sequences of length $2$;	
	\item for every pairwise distinct $s_0,\ldots,s_{2w_n+1}\in 2^n$ we can find $F=\{2+2p,3+2p\}$, for some $p<\left(\genfrac{}{}{0pt}{}{2^n}{2w_n+2}\right)$, such that: 
$$\left\{(\bar{e}^{s_j}_i)_{i\in F}:\ j\in 2w_n+2\right\}=$$
$$=\{(0,0)\}\cup\{(j\cdot 3^{k_{n-1}+2+2p-l},25\cdot 3^{k_{n-1}+2p+3}-1-j):\ j\in 2w_n+1\},$$
where $l$ is such that $w_n=m_l$.
\end{itemize}
Let $E$ consist of all points of the form $\sum_{n\geq 1}e_{x|n}$, for some $x\in 2^\omega$, where $e_{s}=\sum_{i=0}^{u_n-u_{n-1}-1}\frac{e^s_{i}}{q_{k_{n-1}+i}}$, for $s\in 2^n$. 

Now, given any infinite $\{(x_j):\ j\in\omega\}\subseteq 2^\omega$, denote $e_j=\sum_{n\geq 1}e_{x_j|n}\in E$, for each $j\in\omega$. Define $\tilde{n}$, $\tilde{k}$, $j_0$ and $j_1$ similarly as in the previous step. Again, we have:
$$(B-e_{j_0})\cap(B-e_{j})=\left(\bigcup_{l\leq \tilde{k}}(X_l-e_{j_0})\right)\cap\left(\bigcup_{l\leq \tilde{k}}(X_l-e_{j})\right),$$
whenever $x_{j}|\tilde{n}=x_{j_1}|\tilde{n}$. Indeed, 
$$\left(\bigcup_{l>\tilde{k}}X_l\right)-\frac{1}{q_{\tilde{k}}}\subseteq\left\{\phi((x_i)_i):\ (x_i)_i\in S\ \wedge\ \forall_{l<\tilde{k}}x_{l}\in L_{l}\ \wedge\ x_{\tilde{k}}\in L_{\tilde{k}}-1\right\}.$$
Denote the latter set by $Y'$ and observe that $Y'\cap\bigcup_{l>\tilde{k}}X_l=\emptyset$ (as $L_{\tilde{k}}$ does not contain any two consecutive integers) and $Y'\cap\bigcup_{l\leq \tilde{k}}(B\cap X_l)=\emptyset$, since 
$$\bigcup_{l\leq \tilde{k}}(B\cap X_l)\subseteq \left\{\phi((x_i)_i):\ (x_i)_i\in S\ \wedge\ \exists_{l<\tilde{k}}x_{l}\notin L_{l}\right\}\cup$$
$$\cup\left\{\phi((x_i)_i):\ (x_i)_i\in S\ \wedge\ \forall_{l<\tilde{k}}x_{l}\in L_{l}\ \wedge\ \text{dist}(x_{\tilde{k}},L_{\tilde{k}})\geq \lfloor\frac{m_{\tilde{k}}}{2}\rfloor-1\right\}.$$
Therefore, $\left(\bigcup_{l>\tilde{k}}X_l\right)-\frac{1}{q_{\tilde{k}}}$ does not intersect $B$. What is more, the above reasoning shows also that the distance between those two sets is at least $\frac{1}{q_{\tilde{k}+1}}$. Hence, $\left(\bigcup_{l>\tilde{k}}X_l\right)-d_j$ is disjoint with $B-d_{j'}$ whenever $x_j|\tilde{n}\neq x_{j'}|\tilde{n}$. 

The rest of the proof is exactly the same as in the previous step. This finishes the entire proof.
\end{proof}

The above result can be transferred to the case of null-finite sets. Recall that a subset $A$ of an abelian Polish group $X$ is null-finite if there is a convergent sequence $(x_n)\subseteq X$ such that $\{n\in\omega:\ x_n-x\in A\}$ is finite for all $x\in X$. Clearly, each Haar-finite set is null-finite -- a witness for a null-finite set is compact (i.e., the compact set $\{x_n:\ n\in\omega\}\cup\{\lim_n x_n\}$ has an analogous property to that of witnesses of Haar-finite sets), but it does not have to be uncountable as in the case of Haar-finite sets. Moreover, it is easy to observe that for $A\neq\emptyset$ the witnessing sequence $(x_n)\subseteq X$ has to have infinitely many values.

The question whether Borel null-finite sets form an ideal was posed in the first version of \cite{2} and asked by Banakh during his talk at the conference Frontiers of Selection Principles (Warsaw, 2017). It should be mentioned that each non-discrete metric abelian group $X$ is a union of two null-finite sets. Thus, the restriction to Borel sets in the above question is crucial. 

\begin{cor}
The family of Borel null-finite subsets of $\R$ is not an ideal.
\end{cor}

\begin{proof}
It suffices to consider the set $Y=X\cup -X$, where $X$ is as in the previous proof. Then $Y$ is a union of four Haar-finite (so also null-finite) sets: $A$, $B$, $-A$ and $-B$. However, one can show that $Y$ is not null-finite by setting any injective convergent sequence $(x_n)_n$ and considering two cases: either it has a decreasing subsequence (in this case find $(x_{k_n})_n$ such that $c_0=x_{k_0}-\lim_n x_n\in (0,\frac{m_0}{q_0})$ and $c_n=x_{k_n}-\lim_n x_n\in(0,\frac{1}{q_{i-1}})$, for all $i\in\omega$, $i\neq 0$, and proceed exactly as in the second step above) or an increasing one (in this case $-X$ is used: find $(x_{k_n})_n$ such that $-c_0=x_{k_0}-\lim_n x_n\in(-\frac{m_0}{q_0},0)$ and $-c_n=x_{k_n}-\lim_n x_n\in(-\frac{1}{q_{i-1}},0)$, for all $i\in\omega$, $i\neq 0$, find $r\in\bigcap_{i\in\omega}(X-c_i)$ in the same way as in the second step above, and observe that $-r\in\bigcap_{i\in\omega}(-X-(-c_i))$ is as needed).
\end{proof}

It is possible to construct a Haar-finite subset of $\R$ which is not Haar-$n$, for any $n\in\omega$. Namely, it can be shown that $\bigcup_{l\in\omega}(C_l+2l)$ is Haar-finite. However, we omit this proof, since we have the following compact example.

\begin{cor}
There is a compact Haar-finite subset of $\R$ which is not Haar-$n$, for any $n\in\omega$. Therefore, $\overline{\h\Fin}\setminus\bigcup_{n\in\omega}\h n\neq\emptyset$.
\end{cor}

\begin{proof}
It suffices to consider the set $A$ from the proof of Theorem \ref{not ideal}. We already know that it is compact and Haar-finite. To show that it is not Haar-$n$, for any $n\in\omega$, observe that $X_n\setminus B\subseteq\overline{X_n\setminus B}\subseteq \overline{X\setminus B}=A$, for each $n$. Therefore, by Lemma \ref{lem1}, it suffices to prove that $C_n\cap I\subseteq X_n\setminus B$ for some open interval $I$ with $C_n\cap I\neq\emptyset$. 

Recall that:
$$X_n=C_n\cap\left(\bigcup_{l\in L_{n-1}}\left[\frac{l}{q_{n-1}},\frac{l+1}{q_{n-1}}\right]\right)\setminus\left(\bigcup_{l'\in L_{n}}\left(\frac{l'}{q_{n}},\frac{l'+1}{q_{n}}\right)\right)$$
and $\max L_n<m_n-1$. Thus, for
$$I=\left\{\phi((x_i)_i)\in [0,1]:\ \forall_{i<n}\ x_i\in L_i\ \wedge\ x_n=2m_n\right\}$$
we have $\emptyset\neq\text{int} I\cap C_n\subseteq X_n$.

Moreover, notice that: 
$$B\cap X_n=\left\{\phi((x_i)_i)\in X_n:\ \text{dist}\left(\{x_n-\lfloor\frac{m_n}{2}\rfloor\},L_n\right)\leq 1\ \vee \right.$$ 
$$\left.\vee\ \text{dist}\left(\{x_{n}+\lfloor\frac{m_n}{2}\rfloor\},L_{n}\right)\leq 1\right\}.$$
Thus, if $\phi((x_i)_i)\in B\cap X_n$, then $x_n\leq \max L_n+\lfloor\frac{m_n}{2}\rfloor+1<\frac{3}{2}m_n$ (the last inequality is due to $\max L_n<m_n-1$). Hence, $(\text{int} I\cap C_n)\cap B=\emptyset$ and we are done.
\end{proof}

Observe that the above example is a countable union of sets, each of which is Haar-$n$ for some $n$. A natural question is whether this is the case for all Haar-finite sets. Actually, we do not know any example denying that observation. 

\begin{question}
\label{question1}
Is it true that $\h\Fin\subseteq\sigma\left(\bigcup_{n\in\omega}\h n\right)$, i.e., is each Haar-finite set a countable union of sets, each of which is Haar-$n$ for some $n$.?
\end{question}

Although, we have a counterexample for the opposite inclusion.

\begin{cor}
A countable union of compact subsets of $\R$, each of which is Haar-$n$ for some $n$, does not have to be Haar-finite, i.e., $\sigma\left(\bigcup_{n\in\omega}\overline{\h n}\right)\setminus\h\Fin\neq\emptyset$.
\end{cor}

\begin{proof}
It suffices to consider the set $X$ from the proof of Theorem \ref{not ideal}. It is not Haar-finite. However, it is a countable union of sets $X_n\in\h (2m_n+1)$, for $n\in\omega$, and the set $\left\{\phi((x_i)_i):\ \forall_{i\in\omega}\ x_i\in L_i\right\}$, which is Haar-$1$ by Proposition \ref{null-1}, as witnessed by the sequence $(\frac{1}{q_n})_n$ (recall that none $L_n$ contains two consecutive integers).
\end{proof}

Recall that $\sigma\left(\bigcup_{n\in\omega}\overline{\h n}\right)\subseteq\h\Ctbl$ (see Theorem \ref{sumyHn}). Thus, Question \ref{question1} would shed some light on the following.

\begin{question}
Is it true that a countable union of Haar-finite sets must be Haar-countable?
\end{question}

\section{Haar-countable sets}

In this section we construct a Haar-countable set which is not Haar-finite and a Haar-null and Haar-meager set which is not a countable union of closed Haar-countable sets. Also, we give a partial answer to a question posed in \cite{2} concerning countable unions of null-finite sets.

\begin{thm}
There is a compact Haar-countable subset of $\R$ which is not Haar-finite. Thus, $\overline{\h\Ctbl}\setminus\h\Fin\neq\emptyset$.
\end{thm}

\begin{proof}
Consider the set $X$ from the proof of Theorem \ref{not ideal}. We have already shown that it is not Haar-finite. Now we will show that the Cantor set $E$ constructed in the last step of the proof of Theorem \ref{not ideal} witnesses that $X$ is Haar-countable. 

Take any uncountable set $\{x^\alpha=(x^\alpha_i)_i:\ \alpha<\omega_1\}\subseteq 2^\omega$ and denote $e_\alpha=\sum_{n\geq 1}\bar{e}_{x^\alpha|n}\in E$, for each $\alpha<\omega_1$. Suppose to the contrary that $z\in\bigcap_{\alpha<\omega_1}(X-e_\alpha)$ and let $m\in\omega$ be minimal such that $x^\alpha_m=0$ for uncountably many $\alpha<\omega_1$ and $x^\alpha_m=1$ for uncountably many $\alpha<\omega_1$. 

Observe first that:
$$\left(\left\{\phi(x_i):\ \forall_{i\in\omega}\ x_i\in L_i\right\}-e_{\alpha_0}\right)\cap\left(\left\{\phi(x_i):\ \forall_{i\in\omega}\ x_i\in L_i\right\}-e_{\alpha_1}\right)=0$$
provided that $x^{\alpha_0}_m=0$ and $x^{\alpha_1}_m=1$ (by the fact that none $L_n$ contains two consecutive integers). Thus, either $z\in\left(\bigcup_{n\in\omega}X_n\right)-e_{\alpha}$ for all $\alpha$ with $x^\alpha_m=0$ or $z\in\left(\bigcup_{n\in\omega}X_n\right)-e_{\alpha}$ for all $\alpha$ with $x^\alpha_m=1$. Now, as there are uncountably many such $\alpha$, we conclude that there is $k$ with $z\in X_k-e_\alpha$ for uncountably many $\alpha$. However, by the definition of $E$, this is impossible, since $\bigcap_{i\leq 2m_k+1}(X_k-e_{\alpha_i})\subseteq \bigcap_{i\leq 2m_k+1}(C_k-e_{\alpha_i})=\emptyset$ for any pairwise distinct $\alpha_0,\ldots,\alpha_{2m_k+1}<\omega_1$. This finishes the proof.
\end{proof}

Now we will construct a null and meager compact set outside $\h\Ctbl$. Actually, the next result was shown by Elekes and Stepr\=ans in \cite[Theorem 1.2]{Elekes} with the use of a set constructed by Erd\H os and Kakutani in \cite{EK} and is an answer to \cite[Problem 3.1]{DK}, which in turn was motivated by Gruenhage (see also \cite{DM}, \cite{ElekesToth}, \cite{MS} and \cite[Proposition 3.11]{Zakrzewski}). However, we provide our own proof, which can be considered more natural and will be needed later in this section. Notice also that for the group $2^\omega$ a far more general fact has been proved by Zakrzewski in \cite[Proposition 3.12]{Zakrzewski}: in every invariant ccc $\sigma$-ideal on $2^\omega$ there is a compact set which is not Haar-countable. 

\begin{thm}
\label{null+meager}
There is a null and meager compact subset of $\R$ which is not Haar-countable. Thus, $\overline{\h\Null\cap\h\M}\setminus\h\Ctbl\neq\emptyset$.
\end{thm}

\begin{proof}
Fix an increasing sequence $(k_n)_n$ of integers with $k_0=0$. Let $q_0=3$ and $q_i=(2n+3)q_{i-1}$ whenever $k_n\leq i<k_{n+1}$. Define $f\colon\omega^\omega\to\mathbb{R}\cup\{\infty\}$ by $f((x_i)_i)=\sum_{i\in\omega}\frac{x_i}{q_i}$. Denote by $X$ the set consisting of all reals of the form $f((x_i)_i)$, where $x_i\in(2n+3)\setminus\{n+1\}$ whenever $k_n\leq i<k_{n+1}$. Obviously, $X$ is compact and nowhere dense. Moreover, if $(k_n)_n$increases sufficiently fast, it is also null.

In order to show that $X$ is not Haar-countable, fix any Cantor set $D$. Without loss of generality we may assume that $\inf D=0$ and $\sup D\leq \frac{2}{3}$ (by translating $D$ and intersecting it with the interval $[0,\frac{2}{3})$). Let $\{d_s:\ s\in 2^{<\omega}\}\subseteq\mathbb{R}_+$ be such that 
$$D=\left\{\sum_{n\in\omega}d_{\alpha|n}:\ \alpha\in 2^\omega\right\}$$
(i.e., if we represent $D$ as an intersection of finite sums of closed intervals, then $d_{s\frown(1)}$, for $s\in 2^{<\omega}$, are the distances between left ends of two intervals from one level lying below the same interval from the previous level). We can additionally assume that $d_s=0$ for all $s\notin S=\{s^\frown (1):\ s\in 2^{<\omega}\}$ (i.e., $S$ is the set of all finite $0-1$ sequences with $1$ at the end) and that $d_{s^\frown(1)}>\sum_{i\in\omega} d_{s^\frown(0)^\frown \bar{1}_i}$, for each $s\in 2^{<\omega}$, where $\bar{1}_i=(1,\ldots,1)\in 2^i$. 

We inductively pick $t_s\in S$, for all $s\in S$. Let $t_{(1)}$ be of the form $(0,\ldots,0,1)$ and such that $d_{t_{(1)}}<\frac{1}{2q_{k_3}}$. Suppose now that $t_s$, for all $s\in S\cap 2^{<n}$, are already defined. For each $s\in 2^{n-1}$ let $t_{s^\frown (1)}$ be such that:
\begin{itemize}
	\item[(a)] $d_{t_{s^\frown (1)}}<\frac{1}{2q_{k_{2^{n+1}+1}}}$;
	\item[(b)] if $s=(0,\ldots,0)$, then $t_{s^\frown (1)}$ is of the form $(0,\ldots,0,1)$;
	\item[(c)] if $s'\in S$ is the longest sequence such that $s'\subseteq s$, then $t_{s^\frown (1)}$ is a concatenation of $t_{s'}$ and some sequence of the form $(0,\ldots,0,1)$.
\end{itemize}

Let also $t_s=\emptyset$ for $s\in 2^\omega\setminus S$ (i.e., $d_{t_{s}}=0$ for $s\in 2^\omega\setminus S$) and define:
$$E=\left\{\sum_{n\in\omega}d_{t_{\alpha|n}}:\ \alpha\in 2^\omega\right\}.$$
Obviously, $E$ is uncountable. By conditions (b) and (c) we also have $E\subseteq D$. Indeed, given any $\alpha\in 2^\omega$ define $L=\{n\in\omega:\ \alpha|n\in S\}$ and let $(l_i)_i$ be an increasing enumeration of $L$. Then for $\alpha'=\bigcup_{i\in\omega}t_{\alpha|l_i}$ (i.e., $\alpha'|\text{lh}(t_{\alpha|l_i})=t_{\alpha|l_i}$ for each $i$) we get $\sum_{n\in\omega}d_{t_{\alpha|n}}=\sum_{n\in\omega}d_{\alpha'|n}\in D$.

It remains to show that $\bigcap_{e\in E}(X-e)\neq\emptyset$. Let $e_\alpha=\sum_{n\in\omega}d_{t_{\alpha|n}}$ and $(e^\alpha_i)_i$ be such that $e^\alpha_i\in 2n+3$, for $k_n\leq i<k_{n+1}$, and $e_\alpha=f(e^\alpha_i)$, for $\alpha\in 2^\omega$.

We will inductively construct a sequence $(x_i)_i\in\omega^\omega$. We start with $x_0=0$. Suppose now that $x_j$, for all $j<i$, are already defined. Let $m$ be such that $k_m\leq i<k_{m+1}$ and define $Z_i$ as the set consisting of all reals of the form $f((y_j)_j)$, where $y_j\in 2n+3$ whenever $k_n\leq j<k_{n+1}$, $y_j\neq 2n+2$ for infinitely many $j\in\omega$, and $y_i=m+1$. Pick $x_i\in(2m+3)\setminus(m+2)$ in such a way that neither $\sum_{j\leq i} \frac{x_j}{q_j}$ nor $\sum_{j\leq i} \frac{x_j}{q_j}+\frac{1}{q_i}$ belongs to the set
$$\bigcup_{\alpha\in 2^\omega}\left(Z_i-\sum_{j\leq i} \frac{e^\alpha_j}{q_j}\right).$$
This is possible by the Pigeonhole Principle. Indeed, observe first that $|(2m+3)\setminus(m+2)|=m+1$ and each set of the form $Z_i-\sum_{j\leq i} \frac{e^\alpha_j}{q_j}$ excludes at most two values of $x_i$. We will show that there are at most $\frac{m}{2}$ pairwise distinct sets of the form $Z_i-\sum_{j\leq i} \frac{e^\alpha_j}{q_j}$. Let $n$ be maximal such that $2^{n}<m$ (so $2^{n+1}\geq m$). Since $k_m\leq i<k_{m+1}\leq k_{2^{n+1}+1}$, by condition (a) we have $d_{t_{s^\frown(1)}}<\frac{1}{2q_{k_{2^{n+1}+1}}}<\frac{1}{2q_i}$, for all $s\in 2^{n-1}$. Hence, $d_{t_{\alpha|(n+1)}}<\frac{1}{2q_i}$ for all $\alpha$. Moreover, $d_{t_{\alpha|(n+j)}}<\frac{1}{2^j q_i}$, for all $\alpha$ and $j\geq 1$, as $q_{j+1}>2q_j$ and $d_{t_{\alpha|(n+j)}}<\frac{1}{2q_{k_{2^{n+j}+1}}}$ by condition (a). It follows that $\sum_{j\geq 1} d_{t_{\alpha|n+j}}<\frac{1}{2q_i}\left(1+\frac{1}{2}+\frac{1}{4}+\ldots\right)=\frac{1}{q_i}$, for all $\alpha$, and $\sum_{j\leq i} \frac{e^\alpha_j}{q_j}=\sum_{j\leq i} \frac{e^\beta_j}{q_j}$ whenever $\alpha|n=\beta|n$. Thus,
$$\left\{\sum_{j\leq i} \frac{e^\alpha_j}{q_j}:\ \alpha\in 2^\omega\right\}=\left\{\sum_{j\leq i} \frac{e^\alpha_j}{q_j}:\ \alpha=s^\frown(0,0,\ldots)\mbox{ for some }s\in 2^{n-1}\right\}.$$
Now it suffices to observe that the latter set has cardinality at most $2^{n-1}<\frac{m}{2}$.

To finish the proof, notice that $x=f((x_i)_i)$ belongs to $\bigcap_{\alpha\in 2^\omega}(X-e_\alpha)$. Indeed, fix any $\alpha$ and suppose to the contrary that $x\notin X-e_\alpha$. Then $x\in\bigcup_{i\in\omega}(Z_i-e_\alpha)$, since $[0,1]\setminus X\subseteq\bigcup_{i\in\omega}Z_i$ and $x+e_\alpha\leq x+\sup D\leq\frac{1}{3}+\frac{2}{3}=1$ (recall that $x_0=0$). Thus, there is $i_0$ such that $x\in Z_{i_0}-e_\alpha$. Observe that 
$$x\in \left[\sum_{j\leq i_0} \frac{x_j}{q_j},\sum_{j\leq i_0} \frac{x_j}{q_j}+\frac{1}{q_{i_0}}\right].$$ 
Denote the above interval by $I$. By the choice of $x_{i_0}$, we know that endpoints of $I$ do not belong to the set $Z_{i_0}-\sum_{j\leq i_0} \frac{e^\alpha_j}{q_j}$, which is a union of intervals of the form $[\frac{p}{q_{i_0}},\frac{p+1}{q_{i_0}})$, for some integer $p$. Thus, the distance between $\max I$ and $Z_{i_0}-\sum_{j\leq i_0} \frac{e^\alpha_j}{q_j}$ is at least $\frac{1}{q_{i_0}}$. Now it suffices to observe that $\sum_{j>i_0} \frac{e^\alpha_j}{q_j}<\frac{1}{q_{i_0}}$. Hence, $I$ is disjoint with $Z_{i_0}-\sum_{j\in\omega} \frac{e^\alpha_j}{q_j}$, a contradiction.
\end{proof}

Trivially, a countable union of Haar-countable sets is Haar-null and Haar-meager. Now we show that the opposite inclusion fails at least in the case of compact Haar-countable sets.

\begin{cor}
There is a compact null and meager subset of $\R$ which is not a countable union of closed Haar-countable sets. Thus, $\overline{\h\Null}\cap\overline{\h\M}\setminus\sigma\overline{\h\Ctbl}\neq\emptyset$.
\end{cor}

\begin{proof}
Consider the set $X$ from the previous proof. We already know that it is null and meager. 

Fix an open interval $I$ with $X\cap I\neq\emptyset$. We will show that $X\cap I\notin\h\Ctbl$. It will follow that $X\notin\sigma\overline{\h\Ctbl}$ (see Remark \ref{sigma}). 

There is $s\in\omega^{<\omega}$ with $s_i\in(2n+3)\setminus \{x+1\}$, for all $i<\text{lh}(s)$, $k_n\leq i<k_{n+1}$, such that $f(s^\frown(0,0,\ldots))\in X\cap I$ (where $(k_n)_n$ and $f$ are such as in the proof of Theorem \ref{null+meager}). Then, for a given Cantor set $D$, without loss of generality we may assume that $\inf D=0$ and $\sup D<\sup(X\cap I)-f(s^\frown(1,0,0,\ldots))$. Let $E$ and $e_\alpha$, for $\alpha\in 2^\omega$, and $Z_i$, for $i\in\omega$, be as in the previous proof. 

In order to show that $\bigcap_{\alpha\in 2^\omega}((X\cap I)-e_\alpha)\neq\emptyset$, construct a sequence $(x_i)_i\in\omega^\omega$ such that $x_i=s_i$ for $i<\text{lh}(s)$, $x_{\text{lh}(s)}=0$ and $x_i$ for $i>\text{lh}(s)$ are defined similarly as in the proof of Theorem \ref{null+meager}. Then $x=f((x_i)_i)$ belongs to $\bigcap_{\alpha\in 2^\omega}((X\cap I)-e_\alpha)$, since otherwise we could find $\alpha\in 2^\omega$ such that $x\in\bigcup_{i\in\omega}(Z_i-e_\alpha)$ (since  
$$\inf(X\cap I)\leq f(s^\frown(0,0,\ldots))\leq x+e_\alpha\leq f(s^\frown(1,0,0,\ldots))+\sup D<\sup(X\cap I)$$
and $[\inf(X\cap I),\sup(X\cap I)]\setminus X\subseteq\bigcup_{i\in\omega}Z_i$) and by proceeding in the same way as in the proof of Theorem \ref{null+meager} we would get a contradiction.
\end{proof}

The next corollary gives a partial answer to the following problem posed in \cite{2}: is $\h\Null\cap\h\M$ equal to the $\sigma$-ideal generated by null-finite sets?

\begin{cor}
There is a compact null and meager subset of $\R$ which is not a countable union of closed null-finite sets.
\end{cor}

\begin{proof}
Let $X$ be as in the proof of Theorem \ref{null+meager}. We will show that $X\cap I$ is not null-finite, where $I$ is a fixed open interval with $X\cap I\neq\emptyset$. It will follow that $X$ is not a countable union of closed null-finite sets (see Remark \ref{sigma}). 

We start similarly as in the proof of the previous corollary: there is $s\in\omega^{<\omega}$ with $s_i\in(2n+3)\setminus \{x+1\}$, for all $i<\text{lh}(s)$, $k_n\leq i<k_{n+1}$, such that $f(s^\frown(0,0,\ldots))\in X\cap I$ (where $(k_n)_n$ and $f$ are such as in the proof of Theorem \ref{null+meager}). We may additionally assume that $f(s^\frown(0,0,\ldots))>\inf(X\cap I)$ and $f(s^\frown(1,0,0,\ldots))<\sup(X\cap I)$.

Fix any convergent sequence $(x_n)_n$. We may assume that it converges to $0$. 

Suppose first that $(x_n)_n$ contains an infinite decreasing subsequence. In this case it suffices to consider a decreasing subsequence $(x_{k_n})_n\subseteq(x_n)_n$ such that $0< x_{k_n}<\sup(X\cap I)-f(s^\frown(1,0,0,\ldots))$, for all $n$, and: 
\begin{equation}
\label{ee}\left|\left\{n\in\omega:\ x_{k_n}\geq \frac{1}{q_i}\right\}\right|<\frac{m}{2}-1,
\end{equation}
where $k_m\leq i<k_{m+1}$. 

Inductively construct a sequence $(x_i)_i\in\omega^\omega$ starting with $x_i=s_i$, for $i<\text{lh}(s)$, and $x_{\text{lh}(s)}=0$. If $i>\text{lh}(s)$ and $x_j$, for all $j<i$, are already defined, let $m$ be such that $k_m\leq i<k_{m+1}$ and define $Z_i$ as in the proof of Theorem \ref{null+meager}. Pick $x_i\in(2m+3)\setminus(m+2)$ in such a way that neither $\sum_{j\leq i} \frac{x_j}{q_j}$ nor $\sum_{j\leq i} \frac{x_j}{q_j}+\frac{1}{q_i}$ belongs to the set
$$\bigcup_{\alpha\in 2^\omega}\left(Z_i-\sum_{j\leq i} \frac{y^n_j}{q_j}\right),$$
where $(y^n_j)_j$ are such that $y^n_j\in 2l+3$, for $k_l\leq j<k_{l+1}$, and $x_{k_n}=f((y^n_j)_j)$. This is possible by the Pigeonhole Principle, since $|(2m+3)\setminus(m+2)|=m+1$, each set of the form $Z_i-\sum_{j\leq i} \frac{y^n_j}{q_j}$ excludes at most two values of $x_i$ and the set $\left\{\sum_{j\leq i} \frac{y^n_j}{q_j}:\ n\in\omega\right\}$ has cardinality at most $\frac{m}{2}$ (by (\ref{ee}) there are less than $\frac{m}{2}-1$ sequences $(y^n_j)_j$ with $y^n_j|i+1\neq \bar{0}_{i+1}$).

We will show that $x=f((x_i)_i)$ belongs to $\bigcap_{n\in\omega}((X\cap I)-x_{k_n})$. Indeed, otherwise we could find $n\in\omega$ such that $x\notin(X\cap I)-x_{k_n}$. Since $[\inf(X\cap I),\sup(X\cap I)]\setminus X\subseteq\bigcup_{i\in\omega}Z_i$ and 
$$\inf(X\cap I)<f(s^\frown(0,0,\ldots))\leq x+x_{k_n}\leq f(s^\frown(1,0,0,\ldots))+x_{k_n}<\sup(X\cap I),$$
we get that $x\in\bigcup_{i\in\omega}(Z_i-x_{k_n})$. By proceeding in the same way as in the proof of Theorem \ref{null+meager}, we get a contradiction.

If $(x_n)_n$ does not contain an infinite decreasing subsequence, pick an increasing subsequence $(x_{k_n})_n\subseteq(x_n)_n$ such that $0> x_{k_n}>\inf(X\cap I)-f(s^\frown(0,0,\ldots))$, for all $n$, and: 
$$\left|\left\{n\in\omega:\ x_{k_n}\leq -\frac{1}{q_i}\right\}\right|<\frac{m}{2}-1,$$
where $k_m\leq i<k_{m+1}$. 

For each $i\in\omega$ define $Z'_i$ as the set consisting of all reals of the form $f((z_j)_j)$, where $z_j\in 2n+3$ whenever $k_n\leq j<k_{n+1}$, $z_j\neq 0$ for infinitely many $j\in\omega$, and $z_i=m+1$. Similarly as in the previous case, inductively construct a sequence $(x_i)_i\in\omega^\omega$ such that $x_i=s_i$, for $i<\text{lh}(s)$, $x_{\text{lh}(s)}=0$ and $x_i\in(2m+3)\setminus(m+2)$, for $i>\text{lh}(s)$ and $k_m\leq i<k_{m+1}$, is such that neither $\sum_{j\leq i} \frac{x_j}{q_j}$ nor $\sum_{j\leq i} \frac{x_j}{q_j}+\frac{1}{q_i}$ belongs to the set
$$\bigcup_{\alpha\in 2^\omega}\left(Z'_i+\sum_{j\leq i} \frac{y^n_j}{q_j}\right),$$
where $(y^n_j)_j$ are such that $y^n_j\in 2l+3$, for $k_l\leq j<k_{l+1}$, and $-x_{k_n}=f((y^n_j)_j)$. This is possible for the same reason as in the previous case (as each set $Z'_i-\sum_{j\leq i} \frac{y^n_j}{q_j}$ excludes at most two values of $x_i$).

In order to show that $x=f((x_i)_i)$ belongs to $\bigcap_{n\in\omega}((X\cap I)-x_{k_n})$, observe that $[\inf(X\cap I),\sup(X\cap I)]\setminus X\subseteq\bigcup_{i\in\omega}Z'_i$ and 
$$\inf(X\cap I)\leq f(s^\frown(0,0,\ldots))+x_{k_n}\leq x+x_{k_n}\leq f(s^\frown(1,0,0,\ldots))<\sup(X\cap I).$$
Thus, if $x\notin(X\cap I)-x_{k_n}$, for some $n\in\omega$, then $x\in\bigcup_{i\in\omega}(Z'_i-x_{k_n})$, which leads to a contradiction. Indeed, let $i_0$ be such that $x\in Z'_{i_0}-x_{k_n}$. Observe that 
$$x\in \left[\sum_{j\leq i_0} \frac{x_j}{q_j},\sum_{j\leq i_0} \frac{x_j}{q_j}+\frac{1}{q_{i_0}}\right].$$ 
Denote the above interval by $J$. By the choice of $x_{i_0}$, we know that endpoints of $J$ do not belong to the set $Z'_{i_0}+\sum_{j\leq i_0} \frac{y^n_j}{q_j}$, which is a union of intervals of the form $(\frac{p}{q_{i_0}},\frac{p+1}{q_{i_0}}]$, for some integer $p$. Thus, the distance between $\max J$ and $Z'_{i_0}+\sum_{j\leq i_0} \frac{y^n_j}{q_j}$ is at least $\frac{1}{q_{i_0}}$. Now it suffices to observe that $\sum_{j>i_0} \frac{y^n_j}{q_j}<\frac{1}{q_{i_0}}$. Hence, $J$ is disjoint with $Z_{i_0}+\sum_{j\in\omega} \frac{y^n_j}{q_j}$, a contradiction.
\end{proof}

We end with an open question. We believe that it requires developing new methods (besides the ones used in this paper).

\begin{question}
Is the family of all Haar-countable sets an ideal?
\end{question}

\subsection*{Acknowledgment}

The author would like to thank Jaros\l aw Swaczyna from \L\'od\'z University of Technology for interesting him in the problem whether Haar-finite sets form an ideal, many valuable discussions and sharing unpublished results of the paper \cite{1} (which had been developed at the same time as this paper). Moreover, Jaros\l aw Swaczyna had a big impact on the proof of Proposition \ref{przeliczalne}.

Also, the author would like to express his gratitude to the anonymous referee for his helpful comments that improved the quality of the manuscript.

\end{document}